\documentclass[12pt,oneside]{amsart}
\usepackage{amssymb}
\usepackage{verbatim}
\usepackage{xcolor}
\usepackage{graphicx}
\usepackage{overpic}
\usepackage{tikz}
\usetikzlibrary{spy,  calc}
\usetikzlibrary{matrix,positioning,arrows.meta}
\usepackage{float}
\usepackage{adjustbox}
\newenvironment{fitdiagram}{%
  \begin{adjustbox}{max totalsize={\linewidth}{0.86\textheight},center}%
}{\end{adjustbox}}
\usepackage[margin=1in]{geometry}
\usepackage{graphicx}
\usepackage{tikz}
\usetikzlibrary{matrix,arrows.meta}

\newtheorem{theorem}{Theorem}[section]
\newtheorem{proposition}[theorem]{Proposition}
\newtheorem{lemma}[theorem]{Lemma}

\newtheorem*{main theorem}{Main Theorem}

\newtheorem{claim}[theorem]{Claim}

\theoremstyle{definition}
\newtheorem{definition}[theorem]{Definition}

\makeatletter
\newtheorem*{rep@theorem}{\rep@title}
\newcommand{\newreptheorem}[2]{%
\newenvironment{rep#1}[1]{%
 \def\rep@title{#2 \ref{##1}}%
 \begin{rep@theorem}}%
 {\end{rep@theorem}}}
\makeatother

\newreptheorem{theorem}{Theorem}
\newreptheorem{corollary}{Corollary}

\theoremstyle{remark}
\newtheorem{remark}[theorem]{Remark}

\usepackage{amsmath}

\numberwithin{equation}{section}

\newcommand{\R}{$R$}

\begin{document}

\title{On a family of stably handle-slide trivial \R-links}
\author[Wenjie Diao]{Wenjie Diao}
\address{\hskip-\parindent
Wenjie Diao\\
School of Mathematical Sciences, Peking University
}
\email{740997225@qq.com}

\begin{abstract}
 An \R-link is an $n$-component link $L$ in $S^3$ such that Dehn surgery on $L$ yields $\#^n S^1\times S^2$. Adapting work of Gompf, Scharlemann and Thompson,  Meier and Zupan produced a family of \R-links $L(p,q; c/d)$, where the pairs $(p,q)$ and $(c,d)$ are relatively prime and $c$ is even.  For most of link in this family, whether it is handle-slide trivial remains open. In this work, we show one subfamily, $L(3,2; 4/d)$, are stably handle-slide trivial for all $d$.
\end{abstract}
\maketitle
\section{Introduction}
An \R-link is an $n$-component link $L\subset S^3$ such that $L$ has a 0-framed Dehn surgery yielding $\#^n (S^1\times S^2)$. 
For the case of \R-knot, Gabai resolved the celebrated Property R Conjecture \cite{Gabai-foliation3}: if $0$-framed Dehn surgery on a knot $K\subset S^3$ yields $S^1\times S^2$, then $K$ is the unknot. 

The Generalized Property R Conjecture (GPRC-see problem 1.10 in \cite{K3}) asks for a corresponding statement for \R-links. More precisely, is it always possible to transform $L$ into the $n$-component unlink by a sequence of handle-slides?
The Generalized Property R Conjecture remains open even for two-component links. 
Gompf, Scharlemann, and Thompson \cite{GST} studied the case of $2R$-links and proposed several potential counterexamples, most of which remain open. This construction was generalized by Meier and Zupan \cite{generalized-square-knot}  to produce more general family of \R-links with more than two components, $L(p,q;c/d)$, with $(p,q)$ and $(c,d)$ relatively prime integers and $c$ even. 
With these parameters, $L(3,2; 2n/(2n+1))$ is stably equivalent (to be defined below) to the Gompf-Scharlemann-Thompson links $L_n$ in \cite{GST}. 

In this work, we are working with a weaker version: the Stable Generalized Property R Conjecture (stable GPRC-see problem 1.10 in \cite{K3}): If $L$ is an \R-link, then there is an unlink $U$ such that $L\sqcup U$ can be converted by handle-slides to an unlink. Here, $L\sqcup U$ denotes adding a split unlink $U$ to $L$.
Then two links are called \emph{stably handle-slide equivalent} or \emph{stably equivalent} if they can be converted to each other by handle-slides with possible addition of unlink. For background material on handle-slides, see \cite{4-manifolds-Akbulut, 4-manifolds-Gompf}.

While a sufficient condition for a 2-component link to satisfy this stable version is given in \cite{classification},
most of the \R-links mentioned above remain potential counterexamples to the Stable Generalized Property R Conjecture.  

Related to these conjectures, the Andrews-Curtis Conjecture (ACC-see Problem 5.10 in \cite{K3}), posits that any balanced presentation of the trivial group can be converted to the trivial presentation through balanced presentations. If  the GPRC is true, then so is the ACC. There is also a stable version of the ACC that is implied by the Stable GPRC. For more details, see  \cite{GST, generalized-square-knot}.

 In \cite{EZ}, the authors showed that the associated balanced group presentations of a subfamily, $L(3,2;4/d)$, are stably Andrews-Curtis trivial. Their result leaves open whether the links themselves admit a stable handle-slide trivialization, a question they explicitly raised in \cite{EZ}, Question 2. Experimental works in \cite{knotsinfiber, Diao-experiment} established stable handle-slide triviality for some particular values of $d$. In this work, we prove the stable handle-slide triviality for this entire family.

\begin{theorem}
Every link $L(3,2; 4/d)$ is stably handle-slide equivalent to unlink.
\end{theorem}

Note that the link $L(p,q; c/d)$ is actually well-defined up to handle-slides. Since we are working with stable handle-slide triviality, we are free to view $L(p,q; c/d)$ as a stable equivalence class.
The proof uses a uniform construction of two unknotting tunnels for suitable 2-component representatives of these stable equivalence classes, i.e., these 2-component representatives are tunnel number two 2\R-links. The essential geometric step is to control the tunnel systems under the Dehn twists that generate the family: we show that the resulting graphs can be simplified by isotopies and edge slides to unknotted graphs. Then we apply a  tunnel-number criterion of Meier, Schirmer, and Zupan \cite{classification} to obtain stable handle-slide triviality.

\section*{Tools Used}
The author used AI only for its assistance in generating the TikZ files for the pictures in this work.

\section*{Acknowledgements}
The author would like to thank the following
people for their helpful discussions on this work: Hao Chen, Haoqian Pan, Jiajun Wang, Junhua Wang, Chunxing Yan, Yanqing Zou, and Alex Zupan.


\section{Preliminaries}
\subsection{Heegaard splittings and tunnel numbers}
In this subsection, we briefly recall Heegaard splittings and tunnel numbers. For more thorough exposition, see \cite{Heegaard-survey}.

A \emph{Heegaard splitting} of a closed, orientable three-manifold $Y$ is a decomposition of $Y$ into a pair of three-dimensional handlebodies.
This definition can be
generalized to compact $Y$ by using compression bodies.
\begin{definition} Let $\Sigma$ be a closed surface. A \emph{compression body} is a compact
three-manifold $H$ constructed from $\Sigma \times I$ by attaching 2–handles to $\Sigma \times \{0\}$ and
capping off any resulting 2-sphere boundary components with 3–balls. We define
$\partial_{+}H = \Sigma\times \{1\}$and $\partial_{-}H = \partial H\setminus \partial_{+}H$.
\end{definition}

\begin{definition} Let $Y$ be a compact, connected, orientable three-manifold. A
genus $g$ \emph{Heegaard splitting} of $Y$ is a triple $(\Sigma,H_1,H_2)$ satisfying the following conditions:
\begin{itemize}
    \item $Y = H_1\cup H_2$;
    \item $H_i$ is a compression body for $i=1,2$; 
    \item $\Sigma = H_1\cap H_2 = \partial_{+}H_1= \partial_{+}H_2$ is a closed orientable surface of genus $g$.
\end{itemize}
The surface $\Sigma$ is said to be a \emph{Heegaard surface} for $Y$. We also call $Y = H_1\cup_{\Sigma} H_2$ a Heegaard splitting of $Y$.
\end{definition}

For a link $L$ in $S^3$, a set of properly embedded arcs $\{\gamma_1, \cdots, \gamma_n\}$ in the link complement $E(L)$ is called an \emph{unknotting tunnel system} if removing a regular neighborhood of them from $E(L)$ is a handlebody.

\begin{definition}
    The \emph{tunnel number} $t(L)$ of $L$ is the minimal number of tunnels in a tunnel system for $L$.
\end{definition}
Since $E(L)\setminus \cup_{i =1}^n \eta(\gamma_i)$ is a handlebody, this tunnel construction naturally gives a Heegaard splitting $H_1\cup H_2$ of $S^3$, where $H_1$ is the handlebody formed by adding these arcs to $\eta(L)$ and $H_2$ is a handlebody homeomorphic to $E(L)\setminus \cup_{i =1}^n \eta(\gamma_i)$. The connected graph $\Gamma$ formed by adding the tunnels to $L$ is called a \emph{spine} of the handlebody $H_1$: $H_1 \setminus \eta(\Gamma)$ is homeomorphic to $\partial H_1 \times I$. It is standard fact that isotoping such graph in $S^3$ and sliding an edge of this graph over some other edges, called an \emph{edge-slide}, still produces a spine for a handlebody $H_1'$ and that its complement in $S^3$ is also a handlebody $H_2'$. Moreover,  $H_1'\cup H_2'$ is isotopic to the original Heegaard splitting $H_1\cup H_2$. 
In general, we call a graph $\Gamma$ in $S^3$ \emph{unknotted} if its complement in $S^3$ is a handlebody. So the tunnel number $t(L)$ of the link is the minimal number of arcs needed to attach to $L$ so that the resulting graph $\Gamma$ is unknotted. Moreover, isotoping and edge-sliding of this graph preserve (un)knottedness. 
We call two graphs \emph{equivalent} if they are related by isotopies and edge-slides and we vaguely use the term \emph{transformation} for an isotopy or edge-slide of a graph. 

In later sections, we will prove certain link in $S^3$ has tunnel number two by isotoping and edge-sliding the graph obtained so that it is easily seen to be unknotted. For a helpful change of perspective of this edge-slide operation, we sometimes view it as a crushing process illustrated as in Figure \ref{fig:crushing and sliding}. A formal definition of edge-slide can be found in \cite{Nelsen-trisection, Heegaard-survey}.

\begin{figure}[htbp]
\centering

\begin{fitdiagram}
\begin{tikzpicture}[
    operation/.style={
        inner sep=0pt,
        outer sep=0pt
    },
    operation arrow/.style={
        -{Stealth[length=3mm,width=2mm]},
        line width=0.8pt,
        shorten <=3mm,
        shorten >=3mm
    }
]

\matrix (sequence) [
    matrix of nodes,
    nodes={operation},
    column sep=14mm,
    row sep=14mm
] {
    \includegraphics[width=0.45\textwidth]{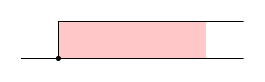}
    &
    \includegraphics[width=0.45\textwidth]{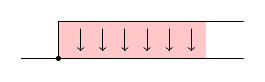}
    \\
    \includegraphics[width=0.45\textwidth]{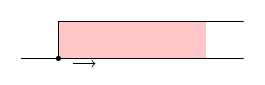}
    &
    \includegraphics[width=0.45\textwidth]{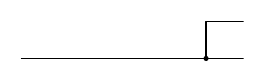}
    \\
};

\draw[operation arrow]
    (sequence-1-1.east) -- (sequence-1-2.west);

\draw[operation arrow]
    (sequence-1-2.south) --node[midway,left] {crushing}(sequence-2-2.north);

\draw[operation arrow]
    (sequence-1-1.south) --(sequence-2-1.north);
\draw[operation arrow]
    (sequence-2-1.east) --node[midway,above] {sliding} (sequence-2-2.west);

\end{tikzpicture}
\end{fitdiagram}

\caption{Illustration for sliding and crushing of a graph.}
\label{fig:crushing and sliding}

\end{figure}

\subsection{The family $L(3,2; 4/d)$}
In \cite{characteizing}, Meier and Zupan used work in \cite{GST, Proposed} to introduce the family $L(3,2; c/d)$ of \R-links. The detailed construction can be found in \cite{characteizing, Proposed}. Here we only briefly recall it and mention some slight distinctions of convention. 

Let $Q_{3,2}$ be the square knot $3_1\# \overline 3_1$, and let $F$ be the fiber surface for $Q_{3,2}$ in $S^3$ with monodromy $\varphi$. It is standard fact that $F$ is a genus two surface. Then $S^3_0(Q_{3,2})$, the closed 3-manifold resulting from 0-surgery on $Q_{3,2}$, is fibered with fiber surface $\hat F$, the closed genus two surface obtained by capping off the boundary component of $F$ with a disk.
Then the closed monodromy map $\hat{\varphi}$ can be easily described as clockwise rotation when $\hat F$ is represented as the quotient of a hexagonally shaped annulus with opposite sides identified in pairs in \cite{Proposed}. We will call such hexagonally shaped annulus \emph{an annulus presentation}.

The $c/d$ parameter comes from the fact that there is a branched covering map $\rho: \hat F\to S$, where $S$ is the 2-sphere with four cone points of order three. Then every simple closed curve on $S$ that avoids the puncture can be represented by an extended rational number $c/d$. For $c/d$ curve in $S$ with $c$ even, it lifts via $\rho$ to three curves $V_{c/d}$, $V'_{c/d}$ and $V''_{c/d}$ in $\hat F$ which are permuted by $\hat \varphi$. Then these curves can be made disjoint from the disk coming from Dehn surgery and thus become embedded in $S^3$. As such the link $Q_{3,2} \cup V_{c/d} \cup V'_{c/d}\cup V''_{c/d}\subset S^3$ is an \R-link, stably equivalent to any 2-component sublink of itself by Lemma 17 of \cite{characteizing}. 
In \cite{generalized-square-knot}, the link $L(3,2;c/d)$ is defined to be $V_{c/d}\cup V'_{c/d}$.
Since we are concerned about whether this link is stably equivalent to unlink, we are free to use any particular element within the stable equivalence class of this link.

As an illustration, we present our dictionary between the hexagonally shaped annulus presentation and $F$ in $S^3$ in Figure \ref{annulus presentation} and Figure \ref{my dictionary}. Concisely, the construction goes briefly as follows: lift the $c/d$ curve on $S$ to annulus presentation via $\rho$; each resulting arc in the annulus presentation can be  classified into the three model arcs (colored yellow, green and blue) in annulus presentation; then translate these arcs back in $F$ and connect them accordingly. For more details, see \cite{Proposed, characteizing, EZ}. We only add a concrete example for $c/d = 1/1$ as an illustration, see Figure \ref{1-1 example}.

\begin{remark}
Our convention only differs by a mirror to that of \cite{Proposed}.
\end{remark}

\begin{figure}
    \centering
    \includegraphics[width=0.4\linewidth]{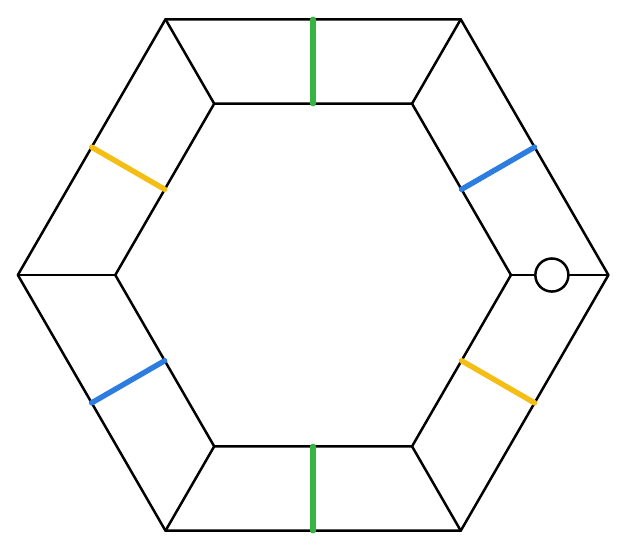}
    \caption{The hexagonally shaped annulus presentation, with yellow curve $V_{0/1}$, green curve $V'_{0/1}$ and blue curve $V''_{0/1}$.}
    \label{annulus presentation}

\end{figure}

\begin{figure}
    \centering
    \includegraphics[width=0.5\linewidth]{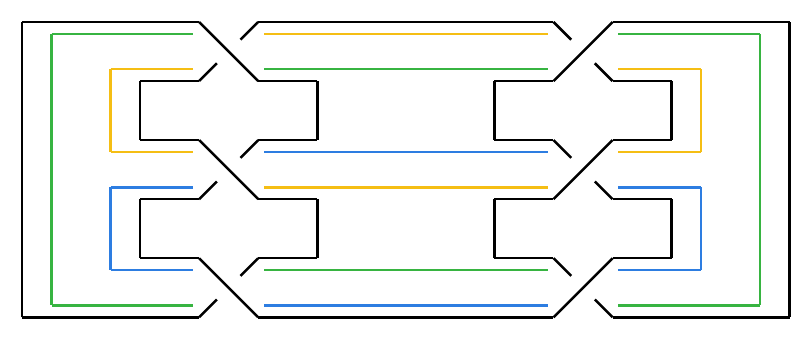}
    \caption{Translating curves on $F$ from annulus presentation.}
    \label{my dictionary}
\end{figure}

\begin{figure}[htbp]
\centering

\begin{fitdiagram}
\begin{tikzpicture}[
    operation/.style={
        inner sep=0pt,
        outer sep=0pt
    },
    operation arrow/.style={
        -{Stealth[length=3mm,width=2mm]},
        line width=0.8pt,
        shorten <=3mm,
        shorten >=3mm
    }
]

\matrix (sequence) [
    matrix of nodes,
    nodes={operation, anchor=center},
    column sep=14mm,
    inner sep=0pt,
    outer sep=0pt
] {
    \includegraphics[width=0.3\textwidth]{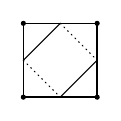}
    &
    \includegraphics[width=0.4\textwidth]{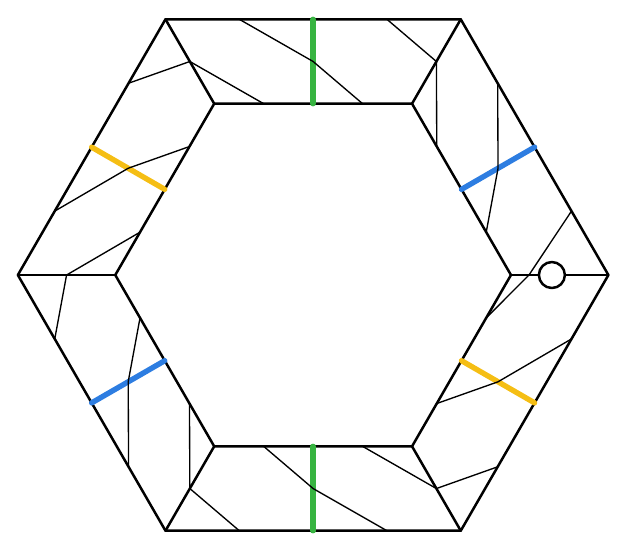}
    \\
};

\node[operation, below=14mm of sequence.south] (result) {
    \includegraphics[width=0.7\textwidth]{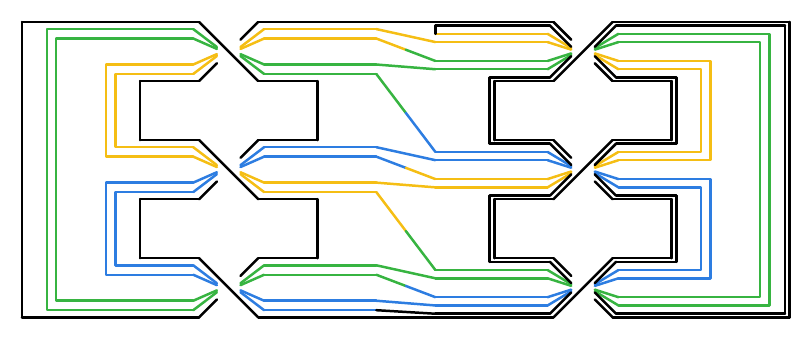}
};

\draw[operation arrow]
    (sequence-1-1.east) -- (sequence-1-2.west);

\draw[operation arrow]
    (sequence-1-2.south)
    -- (result.north);

\end{tikzpicture}
\end{fitdiagram}
\caption{From $1/1$ curve in $S$ to annulus presentation and $F$. The colors are to hint the arc type.}
\label{1-1 example}
\end{figure}

\section{The case $d = 4n+5$}\label{4-5 case}

\begin{figure}
    \centering
    \includegraphics[width=0.5\linewidth]{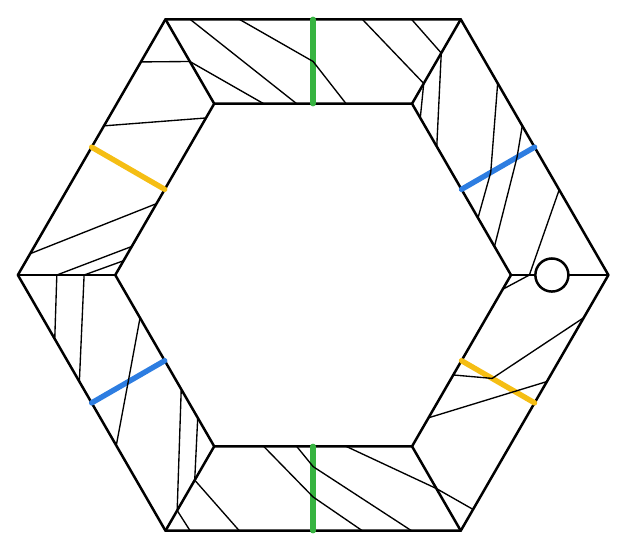}
    \caption{The link $V_{4/5}\cup V'_{4/5}$, with the yellow curve $V_{0/1}$, green curve $V'_{0/1}$ and blue curve $V''_{0/1}$.}
    \label{fig:4/5 link}
\end{figure}

We break the proof of the main theorem into two cases. In this section, we consider the subfamily $L(3,2; 4/(4n+5))$. In the next section, we examine the subfamily $L(3,2; 4/(4n+7))$. The strategies behind both proofs are similar: consider certain 2-component sublinks of $L(3,2; 4/d)$  that are stably equivalent to $L(3,2; 4/d)$ and show that this sublink has tunnel number 2.
By the following theorem, such 2\R-links are stably trivial and so are $L(3,2; 4/d)$:

\begin{figure}[p]
\centering

\begin{fitdiagram}
\begin{tikzpicture}[
    operation/.style={
        inner sep=0pt,
        outer sep=0pt
    },
    operation arrow/.style={
        -{Stealth[length=3mm,width=2mm]},
        line width=0.8pt,
        shorten <=3mm,
        shorten >=3mm
    }
]

\matrix (sequence) [
    matrix of nodes,
    nodes={operation},
    column sep=14mm,
    row sep=14mm
] {
    \includegraphics[width=0.45\textwidth]{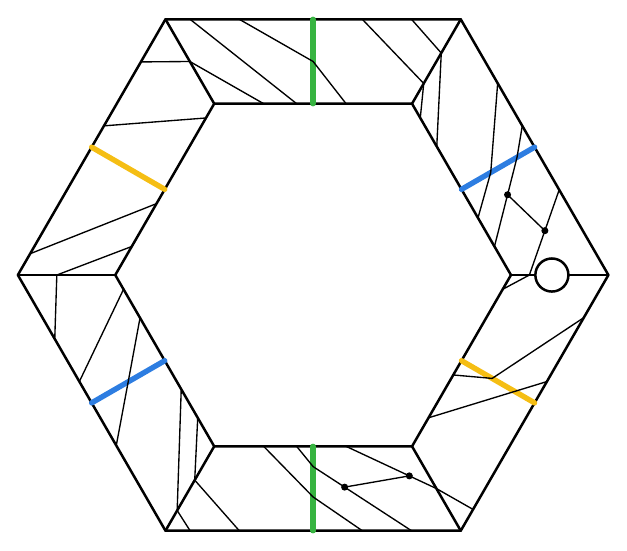}
    &
    \includegraphics[width=0.45\textwidth]{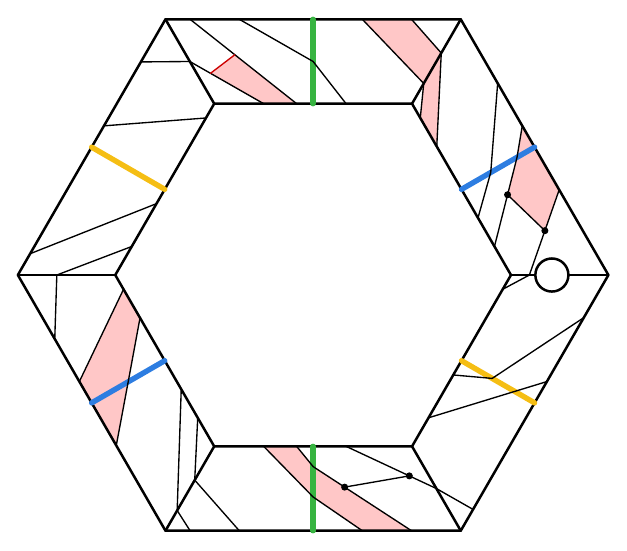}
    \\
    \includegraphics[width=0.45\textwidth]{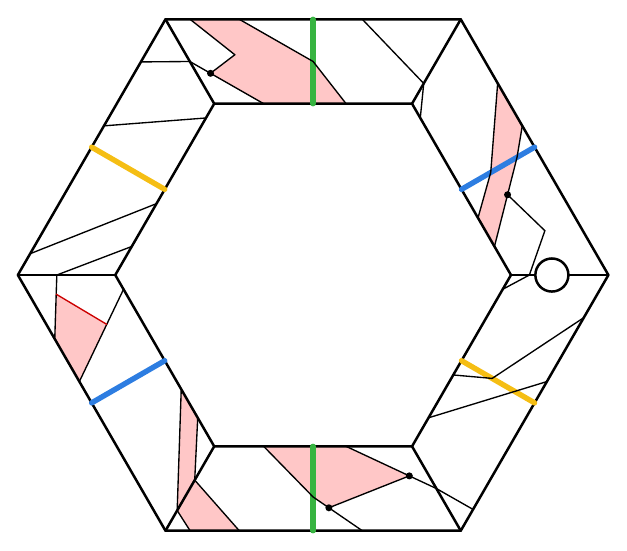}
    &
    \includegraphics[width=0.45\textwidth]{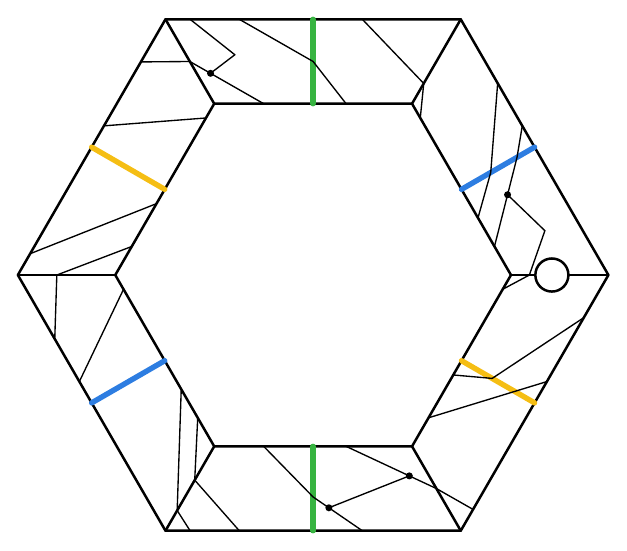}
    \\
    \includegraphics[width=0.45\textwidth]{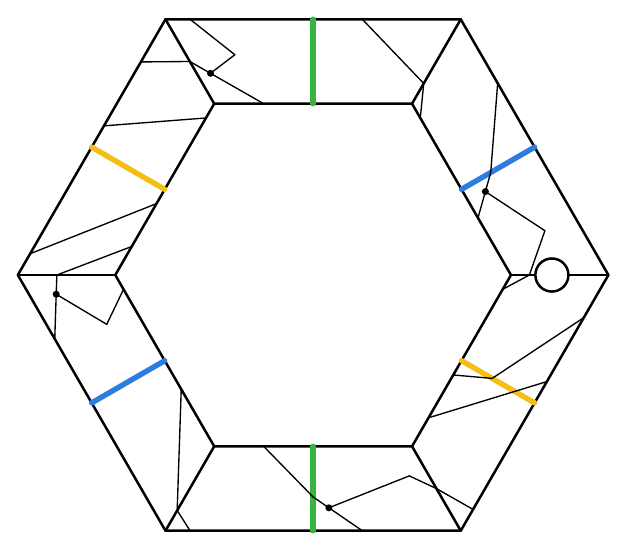}
    &
    \\  
};

\draw[operation arrow]
    (sequence-1-1.east) -- (sequence-1-2.west);

\draw[operation arrow]
    (sequence-1-2.south) --(sequence-2-2.north);

\draw[operation arrow]
    (sequence-2-2.west) --(sequence-2-1.east);

\draw[operation arrow]
    (sequence-2-1.south) --(sequence-3-1.north);


\end{tikzpicture}
\end{fitdiagram}

\caption{A sequence of edge slides and isotopies for 4/5.}
\label{fig:slides for 4/5 in annulus}

\end{figure}

\begin{proposition}[Meier-Schirmer-Zupan \cite{classification}]
Suppose that $L$ is a 2–component link in $S^3$ with an
integral framed surgery to $\#^2 (S^1 \times S^2)$.
 If $L$ has tunnel number 2, then there is a sequence of handle-slides
taking the split union of $L$ with a 0–framed unknot to a
3-component 0–framed unlink.
\end{proposition}

For the remainder of this section, we will label the components in the surface $ {F}$ in Figure \ref{fig:4/5 link} as follows:
\begin{enumerate}
    \item Three colored curves will be $V_{0/1}$, $V'_{0/1}$ and $V''_{0/1}$.
    \item The thin black curves will be $V_{4/5}$ and $V'_{4/5}$. For the purpose of our theorem, we need not distinguish which component exactly is $V_{4/5}$.
\end{enumerate}

Let $ {\tau}:   F\to   F$ be product of the right-handed Dehn twists in the annulus presentation along $ {V_{0/1}}\cup  {V'_{0/1}} \cup  {V''_{0/1}}$.

In Lemma 4.7 and 4.8 of \cite{generalized-square-knot} and Lemma 5 of \cite{EZ}, it was shown the family of link $\{ V_{4/d} \cup V'_{4/d}\}$ has the following inductive relation via $\tau$.
\begin{lemma}\label{Dehn twists of link}
$$
 \tau^n(V_{4/d}) = V_{4/(d+4n)} \text{ and }  \tau^n(V'_{4/d}) = V'_{4/(d+4n)}.
$$
\end{lemma}

Next, we introduce two unknotting tunnels for the link $V_{4/5}\cup V'_{4/5}$, drawn in the left picture in the first row in Figure \ref{fig:slides for 4/5 in annulus}, and simplify the resulting graph using isotopies and edge-slides in this hexagonally shaped annulus  to obtain a graph $\Gamma_{4/5}$, see Figure \ref{fig:slides for 4/5 in annulus}. These simplifications consist of edge-slides and isotopies within the Seifert surface $F$ for the square knot. The simplifications are illustrated using shaded parallel regions in this annulus and we are freely changing perspective for the slide, as in Figure \ref{fig:crushing and sliding}. By Lemma \ref{Dehn twists of link}, the link $V_{4/(5+4n)} \cup V'_{4/(5+4n)} =   \tau^n (V_{4/5}\cup V'_{4/5} )$. So the same picture, namely the left picture in the first row in Figure \ref{fig:slides for 4/5 in annulus}, can be used to represent adding two tunnels to the link $V_{4/(5+4n)} \cup V'_{4/(5+4n)}$. The next Lemma shows that the result of simplifying this graph similarly, denoted by $\Gamma_{4/(5+4n)}$, is the same as Dehn twisting the graph $\Gamma_{4/5}$.

\begin{lemma}\label{Dehn twisting graph}
$$
 \tau^n(\Gamma_{4/5}) = \Gamma_{4/(5+4n)}
$$
\end{lemma}
\begin{proof}
    For the graph obtained by adding two tunnels to the link $V_{4/5}\cup V'_{4/5}$ in the annulus presentation, we consider
    the effect of doing edge-slides indicated by the shaded parallelity regions in Figure \ref{fig:slides for 4/5 in annulus}.  These shaded regions should be thought of as connected, after the identification of opposite edges, and thus can be thought of as a thickened arc in this annulus. Then the indicated sliding can be viewed as crushing this long shaded region into the core arc. Now observe that both copies in the boundary of this thickened arc are either both disjoint or transverse to each Dehn twisting curve. It follows that the crushing process is preserved by Dehn twists. So first doing Dehn twist and then crushing the shaded regions is the same as first crushing the shaded regions and then doing Dehn twist.

\end{proof}

Thus, for each $n\ge 0$, in order to show that the link $V_{4/(5+4n)} \cup V'_{4/(5+4n)}$ has tunnel number 2, it suffices to show that $\Gamma_{4/(5+4n)}$ is unknotted in $S^3$. Having obtained a uniform form of the graph $\Gamma_{4/(5+4n)}$ in the annulus, we now translate it in $S^3$.

In the first picture of Figure \ref{fig:4/5 graph simplify in S3}, we have transformed the graph $\Gamma_{4/5}$ from the annulus into $S^3$, with three Dehn twisting curves drawn. It is worth noting that, in $S^3$ each Dehn twisting becomes left-handed twist, just by following the construction. Moreover, these curves intersect the graph in several points and $\Gamma$ lies in the Seifert surface  $F$ of $Q_{3,2}$. By Lemma \ref{Dehn twisting graph}, the graph  $\Gamma_{4/(5+4n)}$ is the result of twisting $\Gamma_{4/5}$ $n$ times along these curves inside the Seifert surface of $Q_{3,2}$.

In the sequel, we simplify such graph in $S^3$. The idea is to simplify the graph away from the intersection points with each Dehn twisting curve as much as possible. When it is inevitable to consider the effect of the Dehn twists, we show that the twists can, in  many cases, be undone by some further graph simplification process. This allows us to throw some Dehn twisting curves away, i.e., twisting the graph along them does not affect (un)knottedness. Finally, we get a graph with a single Dehn twisting curve in a simple enough position that it is obvious all resulting graphs after twisting are unknotted.  

Firstly, several obvious transformations can be done to simplify the graph as in Figure \ref{fig:4/5 graph simplify in S3}. Note that these transformations are still contained in the Seifert surface. Then we encounter the first circumstance where sliding happens outside the Seifert surface as in Figure \ref{fig:4/5 graph simplify in S3 1}, after which the graph lies back into the Seifert surface again and an obvious simplification appears.

\begin{figure}[htbp]
\centering

\begin{fitdiagram}
\begin{tikzpicture}[
    operation/.style={
        inner sep=0pt,
        outer sep=0pt
    },
    operation arrow/.style={
        -{Stealth[length=3mm,width=2mm]},
        line width=0.8pt,
        shorten <=3mm,
        shorten >=3mm
    }
]

\matrix (sequence) [
    matrix of nodes,
    nodes={operation},
    column sep=14mm,
    row sep=14mm
] {
    \includegraphics[width=0.45\textwidth]{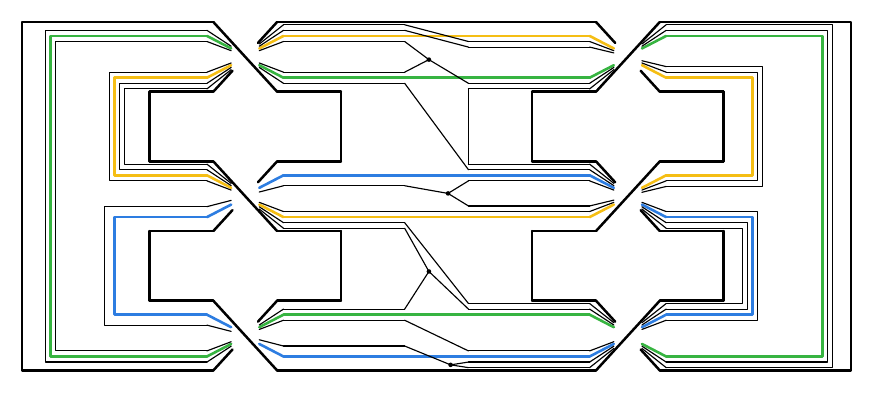}
    &
    \includegraphics[width=0.45\textwidth]{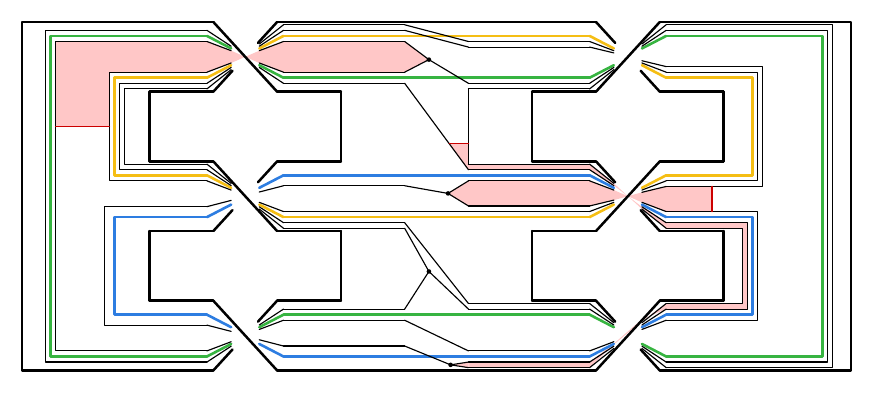}
    \\
    &
    \includegraphics[width=0.45\textwidth]{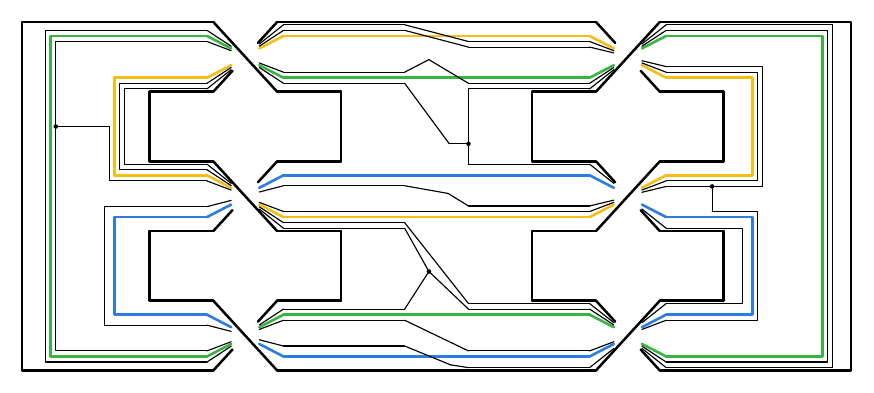}
    \\
  };

\draw[operation arrow]
    (sequence-1-1.east) -- (sequence-1-2.west);

\draw[operation arrow]
    (sequence-1-2.south) -- (sequence-2-2.north);

\end{tikzpicture}
\end{fitdiagram}

\caption{A sequence of edge slides and isotopies inside $F$.}
\label{fig:4/5 graph simplify in S3}

\end{figure}

\begin{figure}[htbp]
\centering

\begin{fitdiagram}
\begin{tikzpicture}[
    operation/.style={
        inner sep=0pt,
        outer sep=0pt
    },
    operation arrow/.style={
        -{Stealth[length=3mm,width=2mm]},
        line width=0.8pt,
        shorten <=3mm,
        shorten >=3mm
    }
]

\matrix (sequence) [
    matrix of nodes,
    nodes={operation},
    column sep=14mm,
    row sep=14mm
] {
    \includegraphics[width=0.45\textwidth]{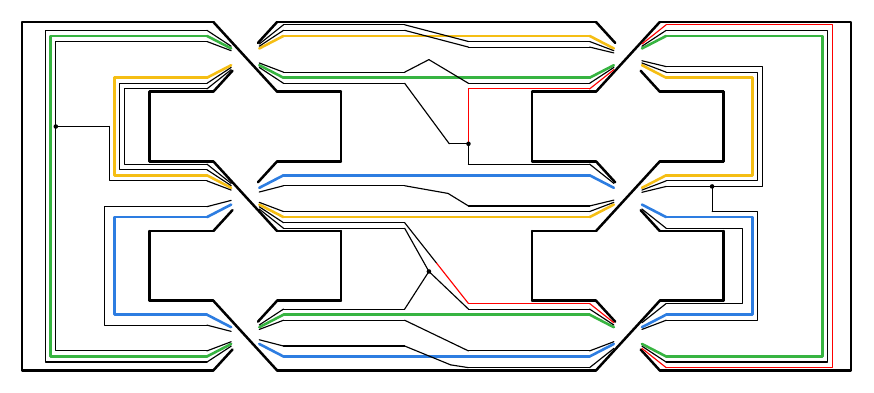}
    &
    \includegraphics[width=0.45\textwidth]{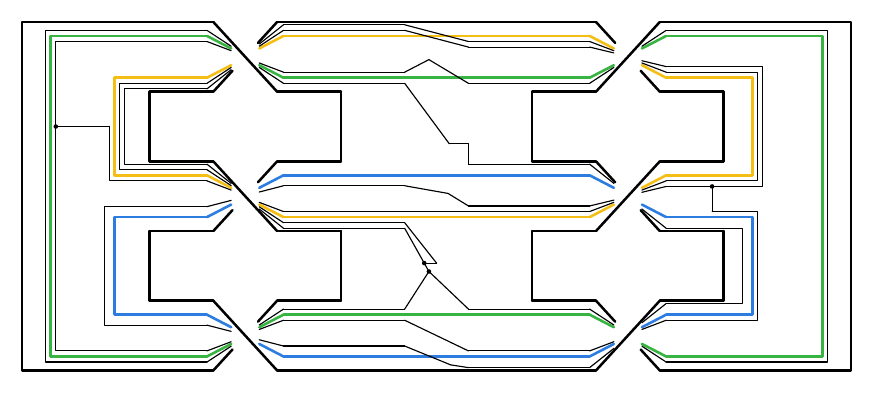}
    \\
  };

\draw[operation arrow]
    (sequence-1-1.east) -- (sequence-1-2.west);

\end{tikzpicture}
\end{fitdiagram}

\caption{A  slide outside $F$: the red arc can be slid, following the edge on the left that goes under every other edge of this graph, to lie in $F$ again.}
\label{fig:4/5 graph simplify in S3 1}

\end{figure}

After such obvious simplification, we arrive at a situation where a further simplification seems to interfere with a twisting curve as in Figure \ref{4/5 graph first twist}. We claim that  we can actually simply pass one strand over this twisting curve:

\begin{figure}[htbp]
\centering

\begin{fitdiagram}
\begin{tikzpicture}[
    operation/.style={
        inner sep=0pt,
        outer sep=0pt
    },
    operation arrow/.style={
        -{Stealth[length=3mm,width=2mm]},
        line width=0.8pt,
        shorten <=3mm,
        shorten >=3mm
    }
]

\matrix (sequence) [
    matrix of nodes,
    nodes={operation},
    column sep=14mm,
    row sep=14mm
] {
    \includegraphics[width=0.45\textwidth]{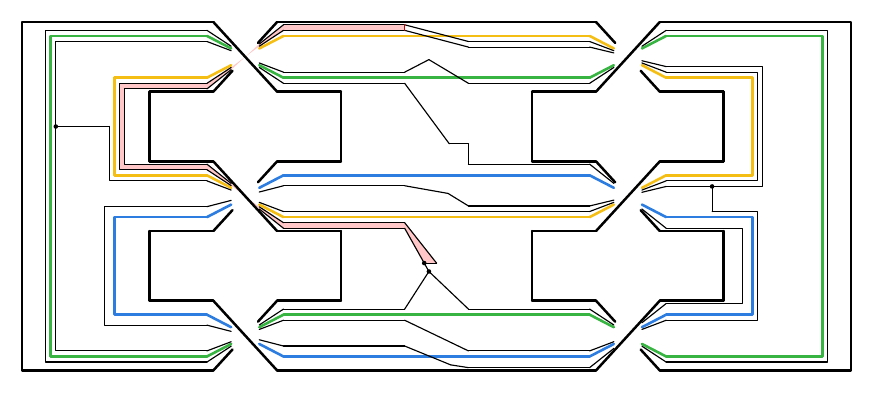}
    &
    \includegraphics[width=0.45\textwidth]{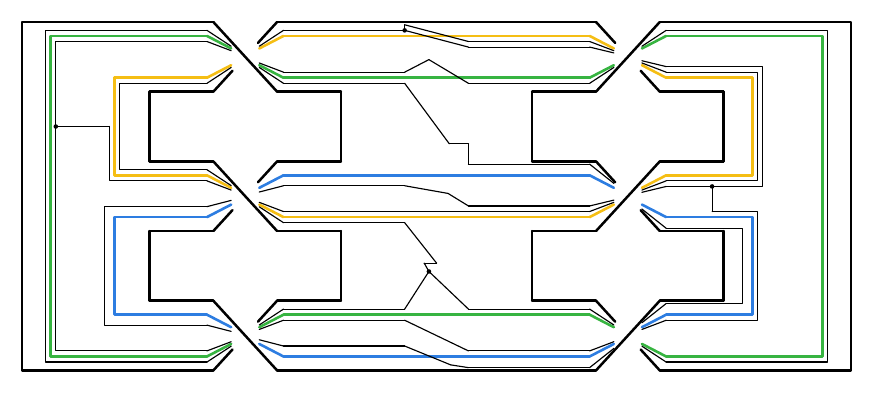}
    \\
    &
    \includegraphics[width=0.45\textwidth]{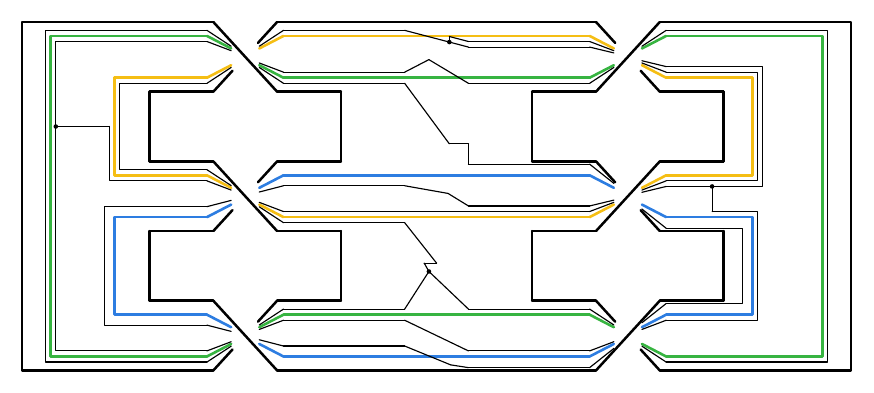}
    \\
  };

\draw[operation arrow]
    (sequence-1-1.east) -- (sequence-1-2.west);

\draw[operation arrow]
    (sequence-1-2.south) -- node[midway,right] {Claim \ref{first twist}}(sequence-2-2.north);

\end{tikzpicture}
\end{fitdiagram}

\caption{A sequence of edge slides and isotopies.}
\label{4/5 graph first twist}

\end{figure}

\begin{figure}[htbp]
\centering

\begin{fitdiagram}
\begin{tikzpicture}[
    operation/.style={
        inner sep=0pt,
        outer sep=0pt
    },
    operation arrow/.style={
        -{Stealth[length=3mm,width=2mm]},
        line width=0.8pt,
        shorten <=3mm,
        shorten >=3mm
    }
]

\matrix (sequence) [
    matrix of nodes,
    nodes={operation},
    column sep=14mm,
    row sep=14mm
] {
    \includegraphics[width=0.3\textwidth]{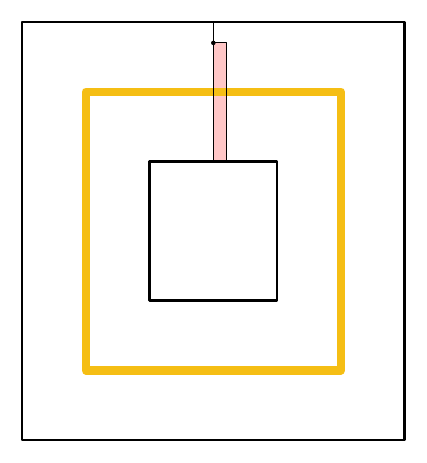}
    &
    \includegraphics[width=0.3\textwidth]{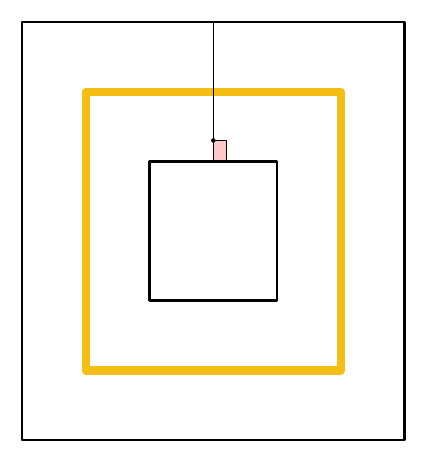}
    \\
};

\draw[operation arrow]
    (sequence-1-1.east) -- node[midway,above] {sliding}(sequence-1-2.west);


\end{tikzpicture}
\end{fitdiagram}

\caption{An edge slide located in an annulus.}
\label{fig:first twist local}

\end{figure}

\begin{claim}\label{first twist}
The graph in the first row on the right is equivalent to the graph in the second row in Figure \ref{4/5 graph first twist}.
\end{claim}

\begin{proof}
This is most easily seen by looking at an annulus tubular neighborhood of this twisting curve in $F$, which is depicted in Figure \ref{fig:first twist local}. Then this graph intersects such annulus in two parallel arcs. The left-handed Dehn twist will simply twist the shaded parallel region around the yellow curve. It follows that sliding the right strand along such twisted region gives the same graph obtained by twisting the graph on the right in Figure \ref{fig:first twist local}.  So  $  \tau^n (\Gamma_{4/5})$ can be edge-slid via this twisted long shaded region within this annulus to pass one strand over the yellow twisting curve.
\end{proof}

Now, after passing through the twisting curve, a new simplification appears as in Figure \ref{fig:third-claim} and we  encounter another crossing of the graph with a twisting curve. We claim that such Dehn twist can be undone:

\begin{claim}\label{second twist}
The graph in the first row on the right is equivalent to the graph in the second row in Figure \ref{fig:third-claim}.
\end{claim}

\begin{proof}
We look at an annulus tubular neighborhood of this Dehn twisting curve again, which is depicted in Figure \ref{fig:second-claim}. Here the only difference with that of Claim \ref{first twist} is that such annulus intersects the graph in a single arc. Back in Figure \ref{fig:third-claim}, part of the graph is disjoint and parallel to the blue twisting curve. So in Figure \ref{fig:second-claim}, we can use such parallel region to slide the strand that crosses the blue curve. Then the Dehn twist along this blue curve is undone via such slide, and this is illustrated in Figure \ref{fig:second-claim}. 
\end{proof}

By Claim \ref{second twist}, in our presentation of the graph $\Gamma_{4/(5+4n)}$ using the graph $\Gamma_{4/5}$ together with Dehn twisting curves, this blue twisting curve can be thrown away. Now we can further simplify as in Figure \ref{fig:operation-sequence3} and we encounter another interference with the green twisting circle on the second row on the left.

\begin{figure}[p]
\centering

\tikzset{
  operation/.style={
    inner sep=0pt,
    outer sep=0pt
  },
  operation arrow/.style={
    -{Stealth[length=3mm,width=2mm]},
    line width=0.8pt,
    shorten <=3mm,
    shorten >=3mm
  }
}

\begin{fitdiagram}
\begin{tikzpicture}
\matrix (sequence) [
  matrix of nodes,
  nodes={operation},
  column sep=14mm,
  row sep=12mm
] {
  \includegraphics[width=0.4\textwidth]{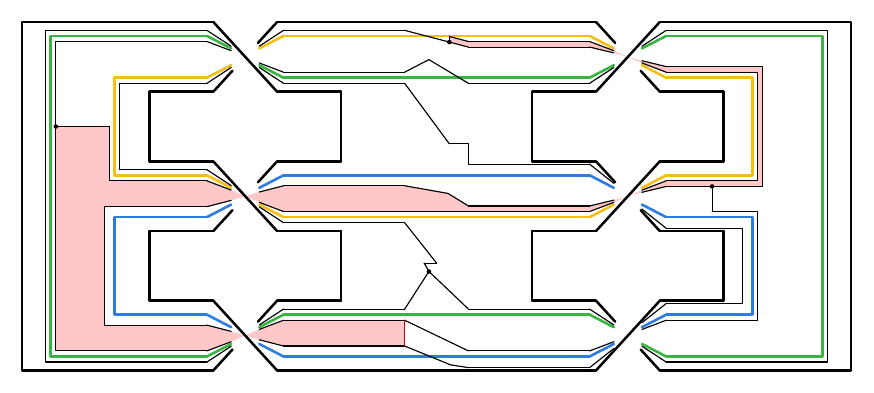}
  &
  \includegraphics[width=0.4\textwidth]{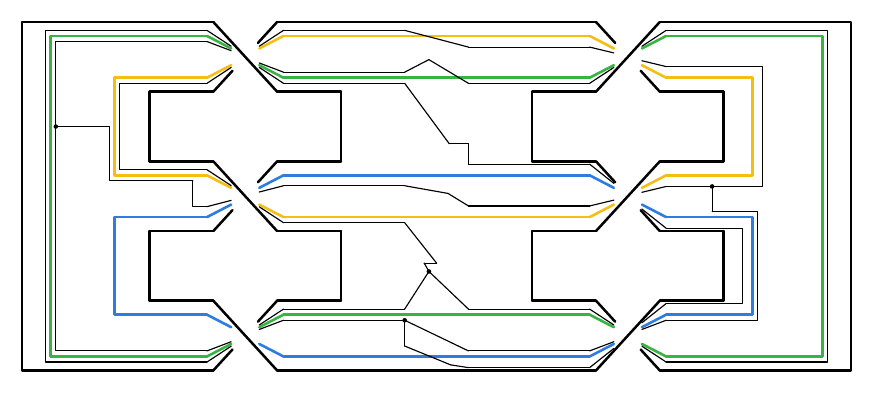}
  \\
  &
  \includegraphics[width=0.4\textwidth]{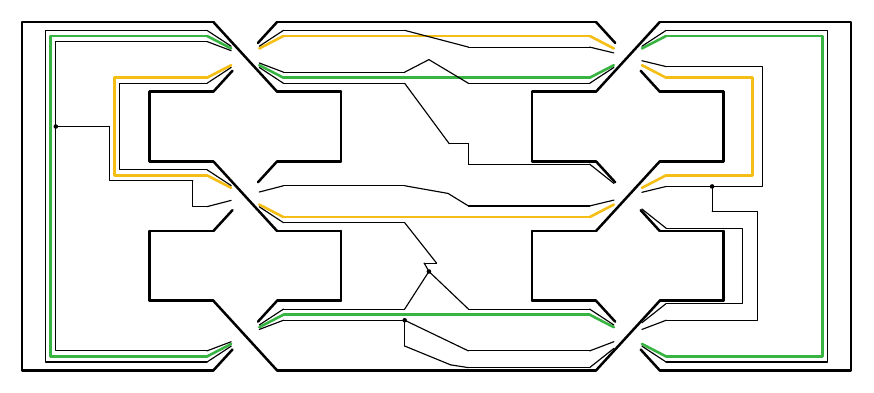}
  \\
};

\draw[operation arrow]
  (sequence-1-1.east)
  --
  (sequence-1-2.west);

\draw[operation arrow]
  (sequence-1-2.south)
  --node[midway,right] {Claim \ref{second twist}}
  (sequence-2-2.north);

\end{tikzpicture}
\end{fitdiagram}

\caption{After a slide, the graph encounters another interference with the twisting curve.}
\label{fig:third-claim}

\vspace{6mm}

\begin{fitdiagram}
\begin{tikzpicture}
\matrix (sequence) [
  matrix of nodes,
  nodes={operation},
  column sep=14mm,
  row sep=12mm
] {
  \includegraphics[width=0.26\textwidth]{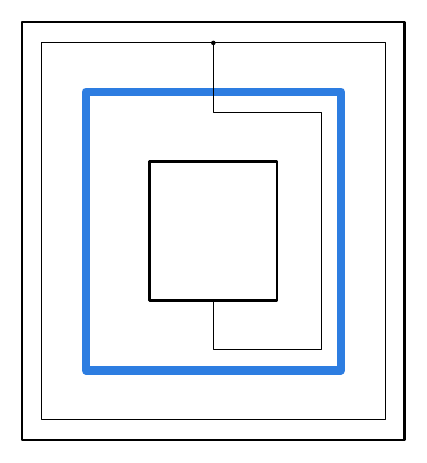}
  &
  \includegraphics[width=0.26\textwidth]{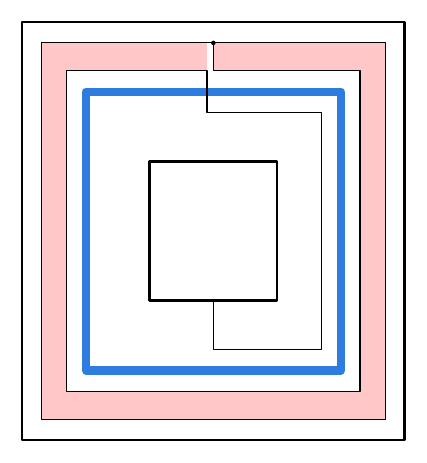}
  \\
  &
  \includegraphics[width=0.26\textwidth]{second_claim_before_twist.pdf}
  \\
};

\draw[operation arrow]
  (sequence-1-1.east)
  -- node[midway,above] {twist}
  (sequence-1-2.west);

\draw[operation arrow]
  (sequence-1-2.south)
  -- node[midway,left] {sliding}
  (sequence-2-2.north);
\end{tikzpicture}
\end{fitdiagram}

\caption{An edge slide located in an annulus.}
\label{fig:second-claim}

\end{figure}


\begin{figure}[p]
\centering

\tikzset{
  operation/.style={
    inner sep=0pt,
    outer sep=0pt,
    minimum width=0.38\textwidth,
    minimum height=0.125\textheight,
    anchor=center
  },
  operation arrow/.style={
    -{Stealth[length=2.5mm,width=1.8mm]},
    line width=0.8pt,
    shorten <=1mm,
    shorten >=1mm
  }
}

\newcommand{\sequenceimage}[1]{%
  \includegraphics[
    width=0.38\textwidth,
    height=0.125\textheight,
    keepaspectratio
  ]{#1}%
}

\begin{fitdiagram}
\begin{tikzpicture}
\matrix (sequence) [
  matrix of nodes,
  nodes={operation},
  column sep=14mm,
  row sep=8mm
] {
  \sequenceimage{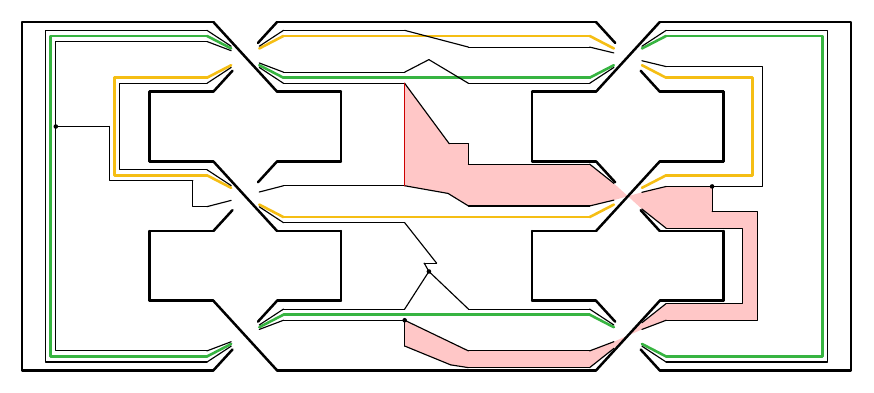}
  &
  \sequenceimage{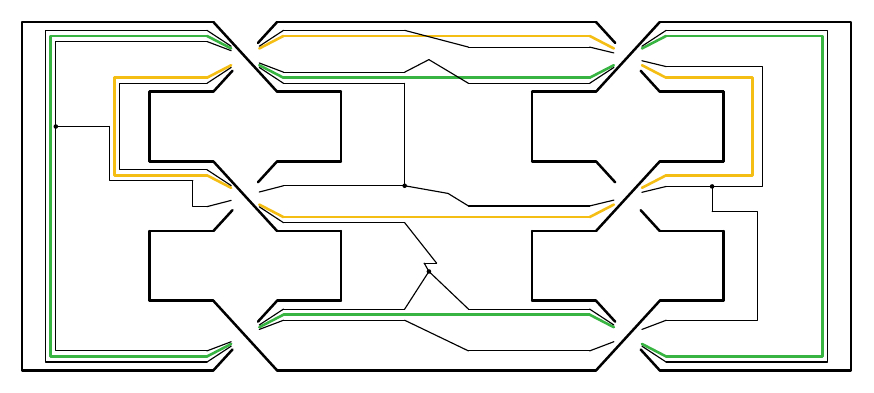}
  \\
  \sequenceimage{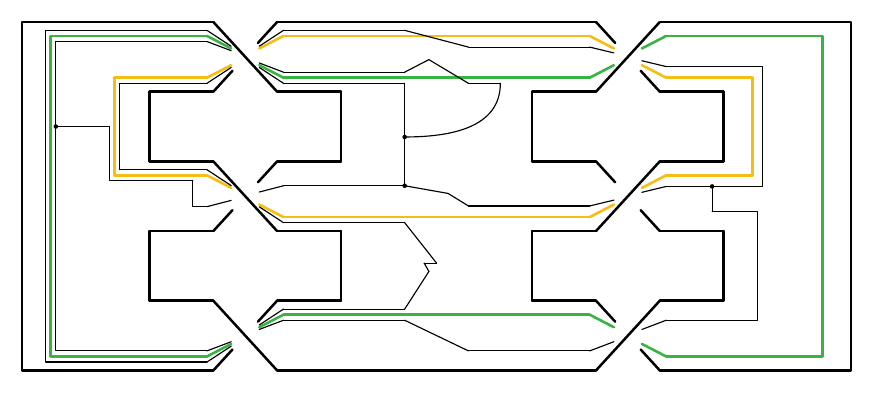}
  &
  \sequenceimage{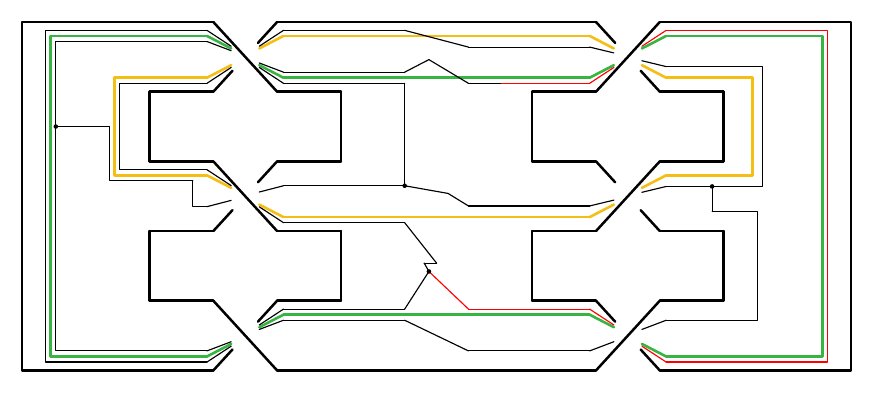}
  \\
  \sequenceimage{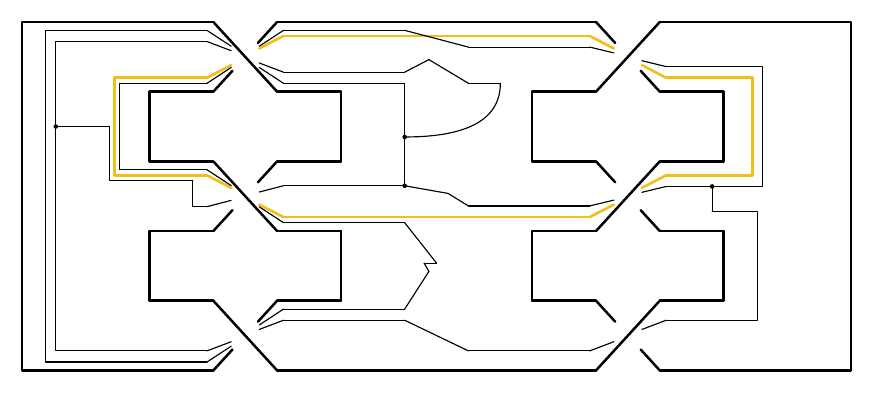}
  &
  \\
};

\draw[operation arrow]
  (sequence-1-1.east) -- (sequence-1-2.west);

\draw[operation arrow]
  (sequence-1-2.south) -- (sequence-2-2.north);

\draw[operation arrow]
  (sequence-2-2.west) -- (sequence-2-1.east);

\draw[operation arrow]
  (sequence-2-1.south)
  -- node[midway,right] {Claim~\ref{third twist}}
  (sequence-3-1.north);

\end{tikzpicture}
\end{fitdiagram}

\caption{A sequence of edge slides and isotopies.}
\label{fig:operation-sequence3}

\vspace{4mm}

\begin{fitdiagram}
\begin{tikzpicture}
\matrix (sequence) [
  matrix of nodes,
  nodes={operation},
  column sep=16mm,
  row sep=8mm
] {
  \sequenceimage{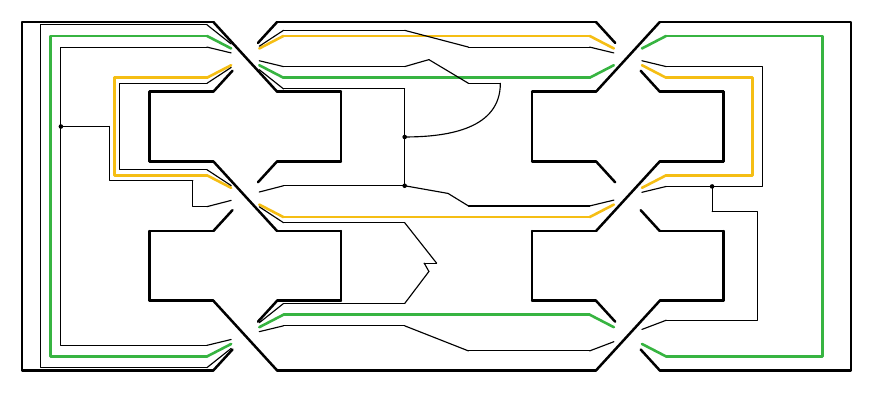}
  &
  \sequenceimage{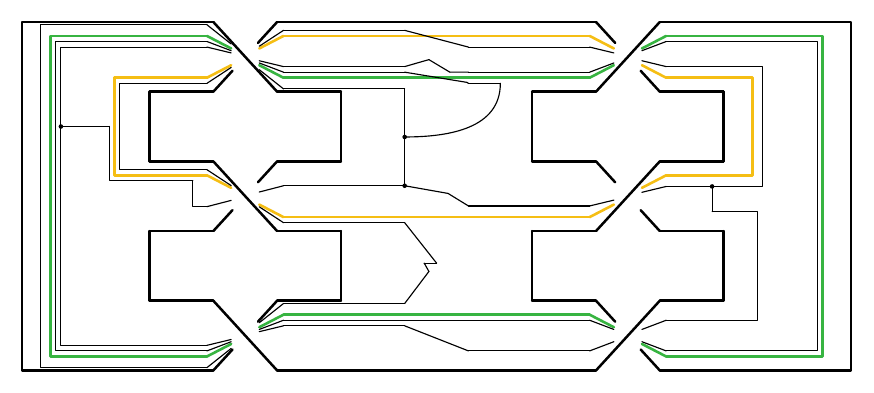}
  \\
   \sequenceimage{4-5_graph_step18.pdf}

  &
  \sequenceimage{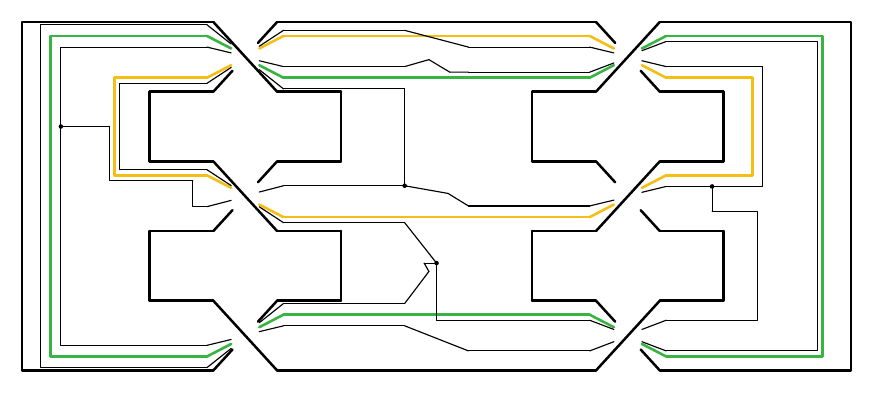}
  \\
};

\draw[operation arrow]
  (sequence-1-1.east) --node[midway, above]{twist} (sequence-1-2.west);

\draw[operation arrow]
  (sequence-1-2.south) --node[midway, right]{slide inside $F$} (sequence-2-2.north);
  
 \draw[operation arrow]
  (sequence-2-2.west) --node[midway, above]{slide}node[midway, below]{outside $F$} (sequence-2-1.east);

\end{tikzpicture}
\end{fitdiagram}

\caption{A sequence of edge slides and isotopies for Claim~\ref{third twist}.}
\label{fig:isolation-third-claim}

\end{figure}

\begin{figure}
\centering

\begin{fitdiagram}
\begin{tikzpicture}[
    operation/.style={
        inner sep=0pt,
        outer sep=0pt
    },
    operation arrow/.style={
        -{Stealth[length=3mm,width=2mm]},
        line width=0.8pt,
        shorten <=3mm,
        shorten >=3mm
    }
]

\matrix (sequence) [
    matrix of nodes,
    nodes={operation},
    column sep=14mm,
    row sep=14mm
] {
    \includegraphics[width=0.45\textwidth]{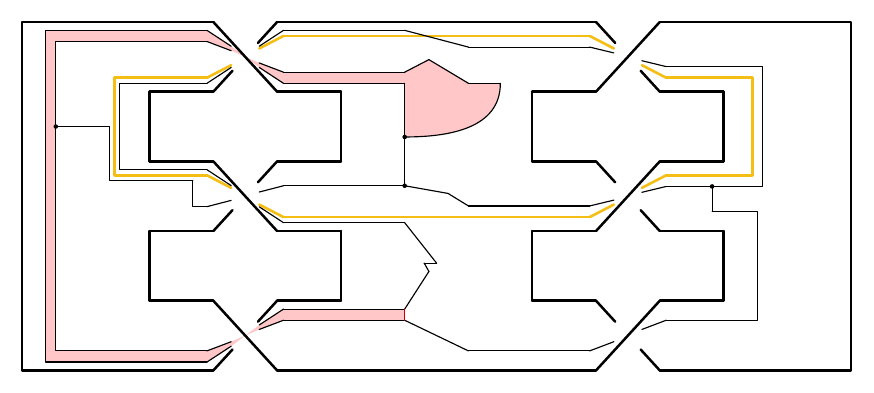}
    &
    \includegraphics[width=0.45\textwidth]{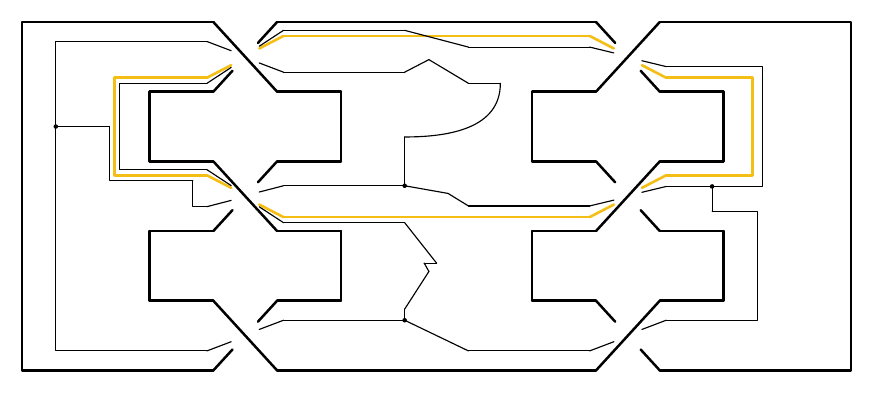}
    \\

    \includegraphics[width=0.45\textwidth]{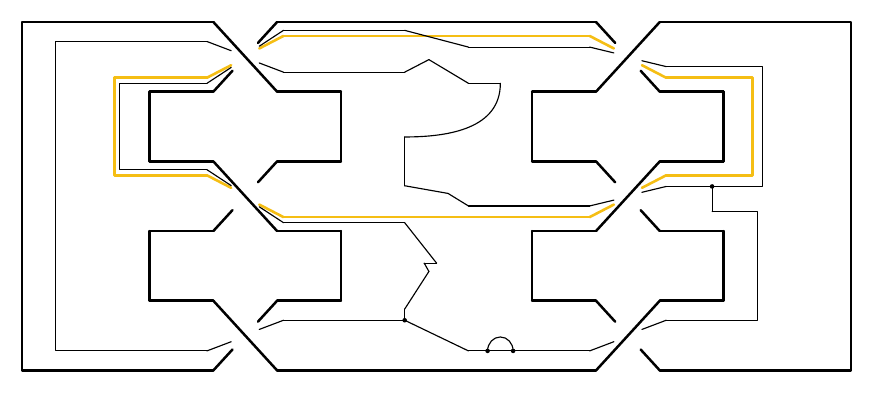}
    &
    \includegraphics[width=0.45\textwidth]{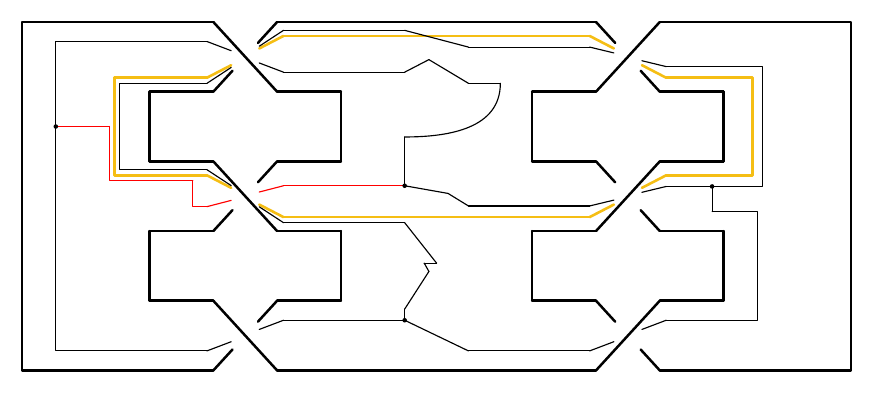}
    \\

    \includegraphics[width=0.45\textwidth]{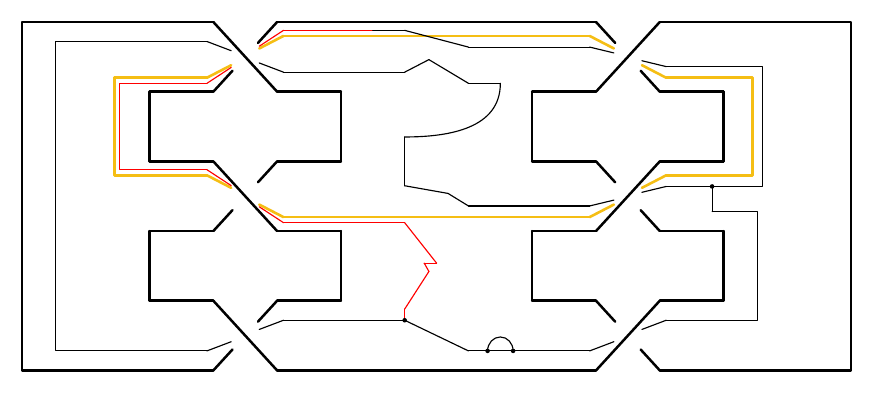}
    &
    \includegraphics[width=0.45\textwidth]{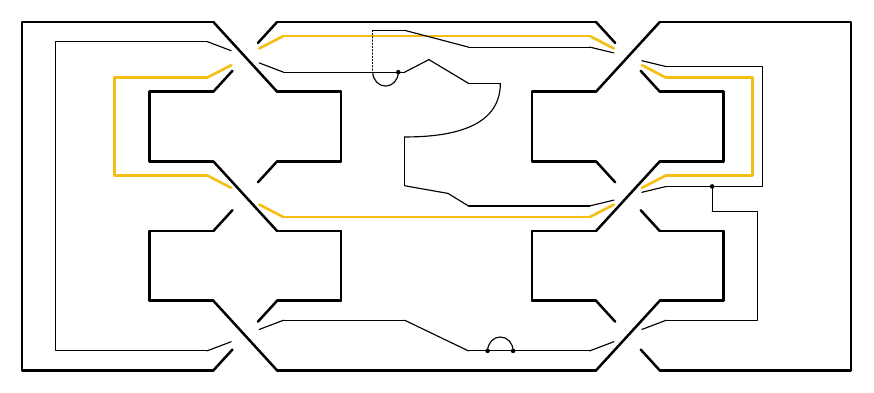}
    \\

    &
    \includegraphics[width=0.45\textwidth]{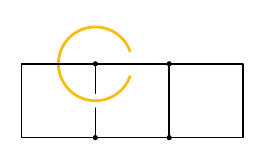}
    \\
    };

\draw[operation arrow]
    (sequence-1-1.east) -- (sequence-1-2.west);

\draw[operation arrow]
    (sequence-1-2.south) --(sequence-2-2.north);

\draw[operation arrow]
   (sequence-2-2.west) -- (sequence-2-1.east);

\draw[operation arrow]
    (sequence-2-1.south) --(sequence-3-1.north);

\draw[operation arrow]
    (sequence-3-1.east) -- (sequence-3-2.west);

\draw[operation arrow]
    (sequence-3-2.south) -- (sequence-4-2.north);

\end{tikzpicture}
\end{fitdiagram}

\caption{A sequence of edge slides and isotopies.}
\label{fig:operation-sequence4}

\end{figure}

We claim next that we can undo such twist.

\begin{claim}\label{third twist}
The graph in the second row on the left is equivalent to the graph in the third row on the left in Figure \ref{fig:operation-sequence3}.
\end{claim}
\begin{proof}
The situation here is just slightly different from before. In an annulus regular neighborhood of this twisting curve, only a single arc crosses the green twisting curve. However, unlike the situation in Claim \ref{second twist}, there isn't part of the graph which is disjoint and parallel to this green circle completely but rather there is just a sub-arc of the graph that is parallel to this green circle. We isolate it in Figure \ref{fig:isolation-third-claim}: still making use of such partial parallel region to slide the strand that crosses the green circle. Then another slide, which is outside of $F$, can be applied to come back to the same graph before the Dehn twist. Inductively, this means that Dehn twists along green circle can be undone via repeated applications of the above two-step process. 
\end{proof}

Claim \ref{third twist} implies that we are free to throw away the green twisting curve, after which we can further simplify as in Figure \ref{fig:operation-sequence4}. After the indicated slides, we obtain the graph in the last row on the right. Since the yellow circle is 0-framed on $F$, it is easy to see that twisting the yellow circle gives unknotted graph in $S^3$. So we have proved that each graph $\Gamma_{4/(5+4n)}$ is unknotted for any $n\ge 0$.

\section{The case $d=4n+7$}
The strategy of proof for this case is almost identical to that of $d=4n+5$ case. The difference lies in the specific positions of the two tunnels. See the first picture in Figure \ref{fig:4-7operation-sequence1}.
\begin{figure}
    \centering
    \includegraphics[width=0.5\linewidth]{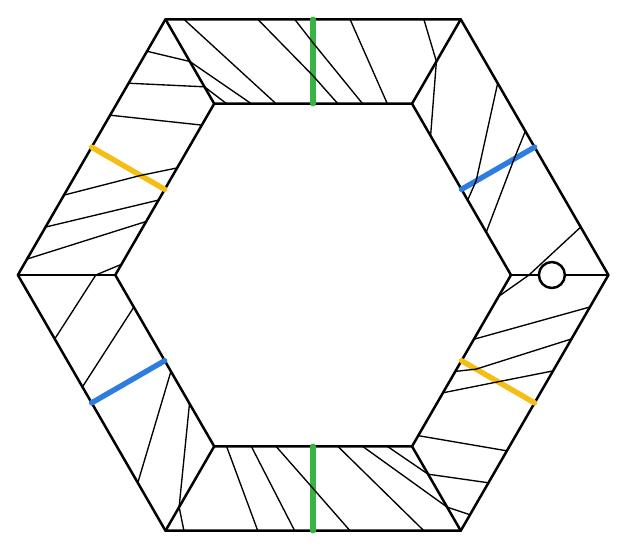}
    \caption{The link for 4/7.}
    \label{fig:placeholder}
\end{figure}

By Lemma \ref{Dehn twists of link}, the link $V_{4/7}\cup V'_{4/7}$ is related to the link $V_{4/(7+4n)}\cup V'_{4/(7+4n)}$ via $n$ times Dehn twist. Thus the two tunnels in the $4/7$ picture in Figure \ref{fig:4-7operation-sequence1}  represent two tunnels for general $4/(7+4n)$ case. We will apply a similar sequence of simplifications to the obtained graph in the hexagonal shaped annulus presentation. Similar to Lemma \ref{Dehn twisting graph}, we have 

\begin{lemma}\label{Dehn twisting graph for 4/7}
$$
 \tau^n(\Gamma_{4/7}) = \Gamma_{4/(7+4n)}
$$
\end{lemma}
The proof is almost identical to the proof of Lemma \ref{Dehn twisting graph}, so we only illustrate this process in Figure \ref{fig:4-7operation-sequence1}.

\begin{figure}[htbp]
\centering

\begin{fitdiagram}
\begin{tikzpicture}[
    operation/.style={
        inner sep=0pt,
        outer sep=0pt
    },
    operation arrow/.style={
        -{Stealth[length=3mm,width=2mm]},
        line width=0.8pt,
        shorten <=3mm,
        shorten >=3mm
    }
]

\matrix (sequence) [
    matrix of nodes,
    nodes={operation},
    column sep=14mm,
    row sep=14mm
] {
    \includegraphics[width=0.45\textwidth]{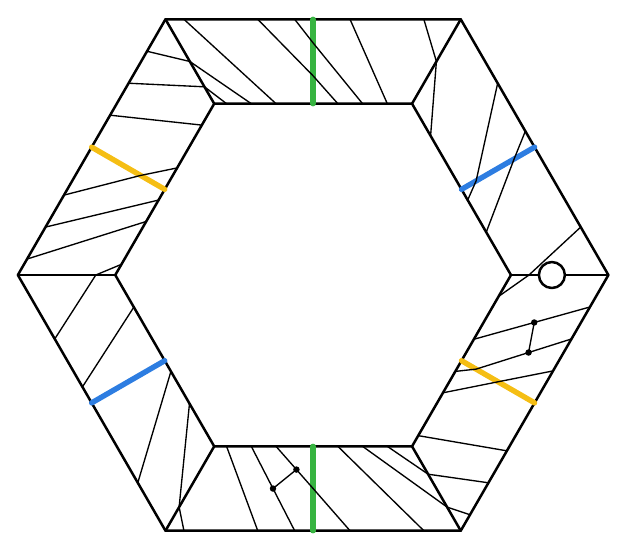}
    &
    \includegraphics[width=0.45\textwidth]{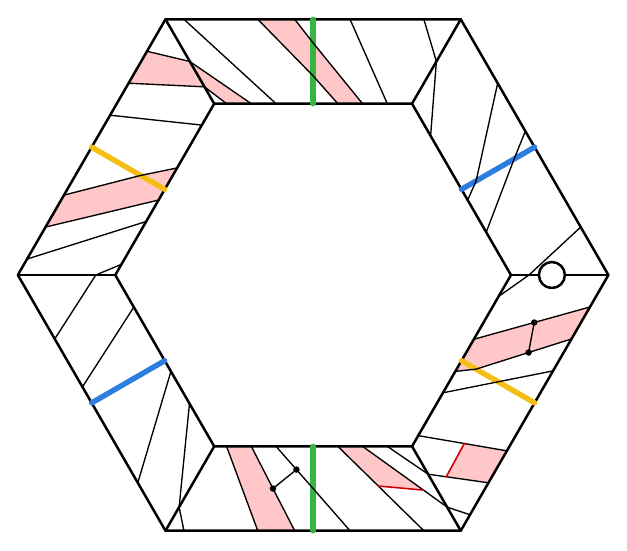}
    \\

    \includegraphics[width=0.45\textwidth]{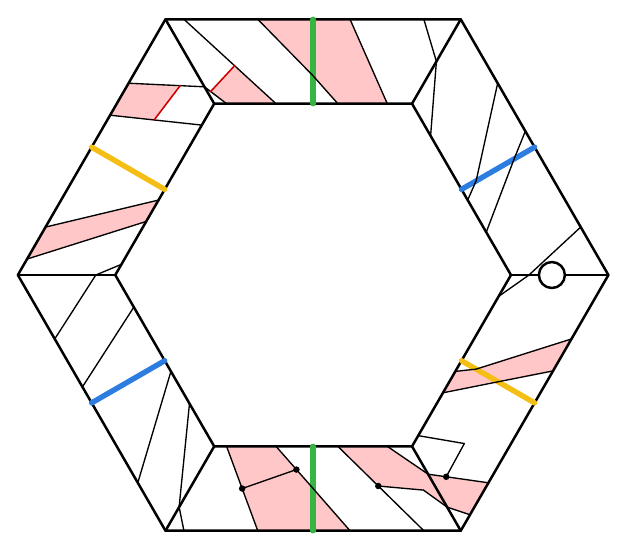}
    &
    \includegraphics[width=0.45\textwidth]{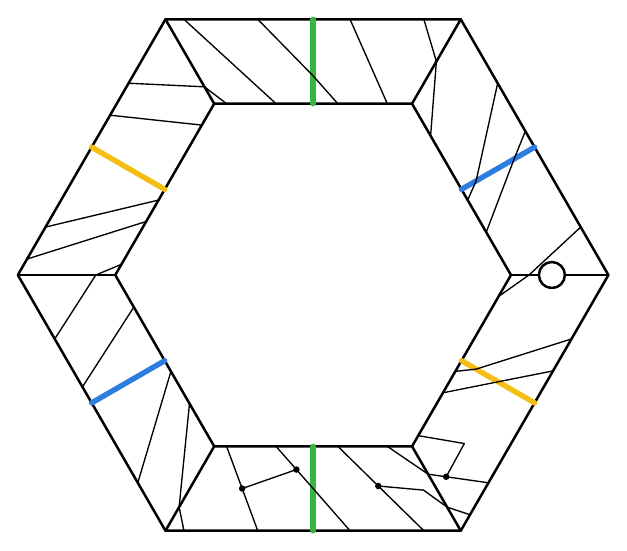}
    \\

    \includegraphics[width=0.45\textwidth]{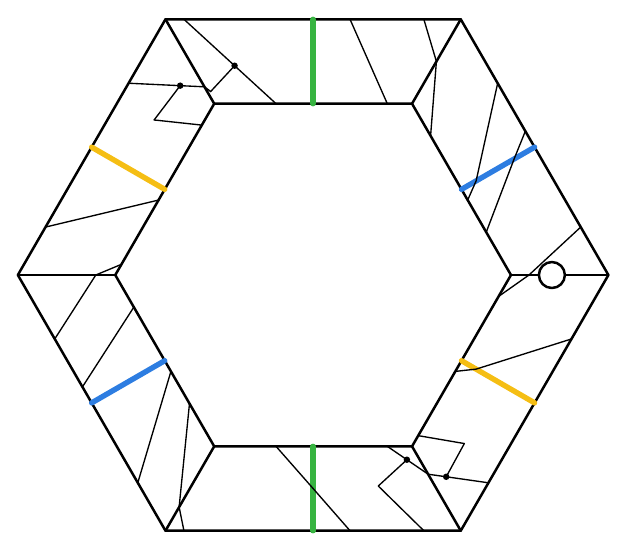}
    &
    \\
};

\draw[operation arrow]
    (sequence-1-1.east) -- (sequence-1-2.west);

\draw[operation arrow]
    (sequence-1-2.south) -- (sequence-2-2.north);

\draw[operation arrow]
    (sequence-2-2.west) -- (sequence-2-1.east);

\draw[operation arrow]
    (sequence-2-1.south) -- (sequence-3-1.north);


\end{tikzpicture}
\end{fitdiagram}

\caption{A sequence of edge slides and isotopies for 4/7.}
\label{fig:4-7operation-sequence1}

\end{figure}

Next, we pass from the annulus presentation to graph in $S^3$, as in  the first picture in Figure \ref{fig:4-7 operation-sequence2}. Then we start to simplify this graph as indicated. As we obtain the graph in the fifth row on the left, we encounter the interference with the blue twisting curves. The situation here is similar to that of Claim \ref{first twist}: there are two intersecting arcs. We isolate this process in Figure \ref{fig:4-7 first twist}: comparing the first picture in the last row and the second picture in the last row of Figure \ref{fig:4-7 operation-sequence2}, this means that twist first and then slide is the same as slide first and then twist. So, just like in the case of Claim \ref{first twist}, we have passed one strand over the twisting curve.

Then we continue the  simplification as indicated in Figure \ref{fig:4-7operation-sequence3}. For the transition of the picture in the first row on the right to the picture in the second row on the right, this time the situation is completely the same as that of Claim \ref{second twist}, so we can throw away the yellow twisting curve. After this, we follow the indicated edge slides and isotopies to meet another interference of the graph with the green curves in the third row: this time in order to undo the twist, we need a repeated application of 2-step process just like in Claim \ref{third twist} and we isolate this process in Figure \ref{fig:4-7 third twist}.

After this, we follow the indicated slides in Figure \ref{fig:4-7operation-sequence3} and in Figure \ref{fig:4-7operation-sequence4}. Finally, we get the last graph in Figure \ref{fig:4-7operation-sequence4}, and it is easy to see that this graph is unknotted for any number of twists. So we have proved that each graph $\Gamma_{4/(7+4n)}$ is unknotted for any $n\ge 0$.

\section{Further discussions}
A trisection is a decomposition of a 4-manifold $X$ into three simple pieces and can be thought of as a 4-dimensional version of a Heegaard splitting of 3-manifold. Gay and Kirby \cite{GayKirby} introduced the theory of trisection of 4-manifolds and a connection between trisection and (stable)GPRC has been studied in \cite{classification, characteizing}. More precisely, it is shown in Theorem 3 of \cite{characteizing} that for any \R-link $L$ and its admissible surface $\Sigma$, $L$ is stably handle-slide trivial if and only if the associated trisection $\mathcal T(L, \Sigma)$ is 2,3-standard. It is also indicated in \cite{EZ} that one approach to show that the links $L(3,2; 4/d)$ are stably handle-slide trivial is to show that resulting genus four trisections $\mathcal T(4/d)$, under certain choice of genus four admissible surface, are all 2,3-standard. As a direct corollary of our main theorem, such genus four trisections $\mathcal T(4/d)$ are all 2,3-standard. 
Since other currently unknown \R-links $L(3,2; c/d)$ also give rise to similar but more complicated genus four trisections and little progress has been made on standardizing genus four trisections, we find it interesting to work out an explicit standardizing process regarding $\mathcal T(4/d)$. An optimistic  goal would be to extract new understandings from it so as to understand more complicated genus four trisections.

Another tempting approach invokes the notion of $h$-genus of links. In \cite{Diaohgenus}, we show $h$-genus three 2\R-links are stably handle-slide trivial. Since tunnel number two 2\R-links all have $h$-genus less than or equal to three, $h$-genus three condition is less restrictive than that of tunnel number two. So another approach to trivialize $L(3,2; c/d)$ could be to show that they have $h$-genus three. Of course, we don't have good reason to eliminate the possibility that more complicated family of $L(3,2; c/d)$ could fail GPRC or stable GPRC. 

We can strengthen our main theorem a little bit: by using the genus three Heegaard surface given by the two tunnels in Section \ref{4-5 case}, we find that the link $V_{4/5}\cup V'_{4/5}$ is handle-slide trivial, i.e., it satisfies GPRC. More generally, experiments imply that the family $V_{4/d}\cup V'_{4/d}$ satisfy GPRC for many low values of $d$.

\begin{figure}[p]
\centering

\begin{fitdiagram}
\begin{tikzpicture}[
    operation/.style={
        inner sep=0pt,
        outer sep=0pt
    },
    operation arrow/.style={
        -{Stealth[length=3mm,width=2mm]},
        line width=0.8pt,
        shorten <=3mm,
        shorten >=3mm
    }
]

\matrix (sequence) [
    matrix of nodes,
    nodes={operation},
    column sep=14mm,
    row sep=10mm
] {
    \includegraphics[width=0.45\textwidth]{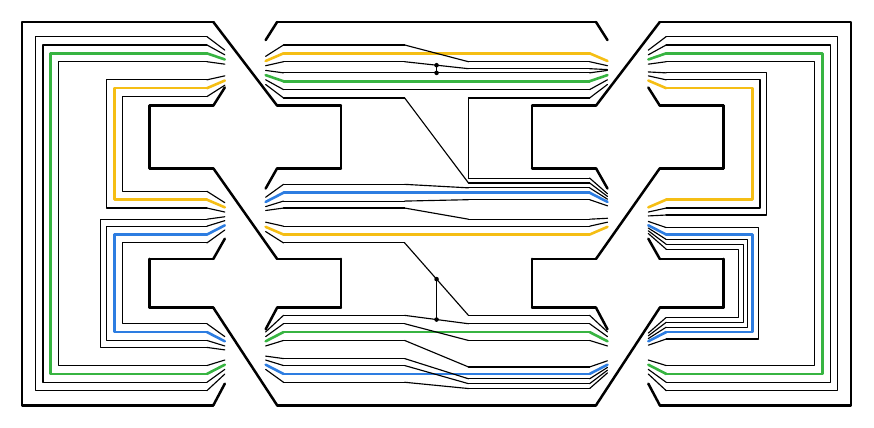}
    &
    \includegraphics[width=0.45\textwidth]{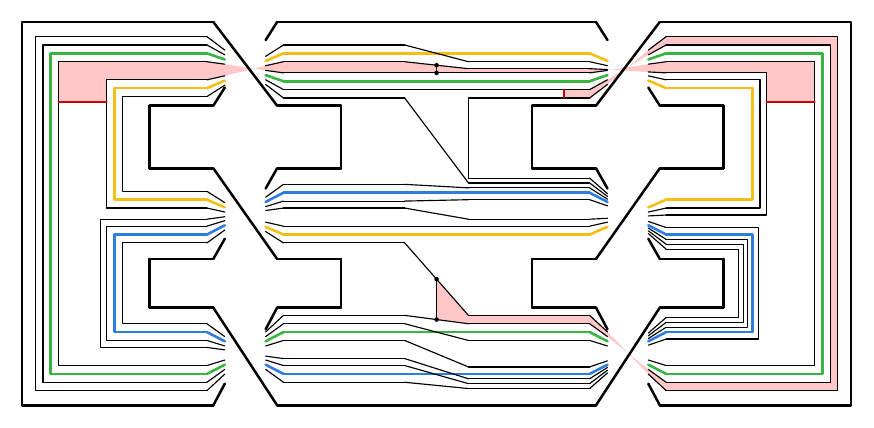}
    \\

    \includegraphics[width=0.45\textwidth]{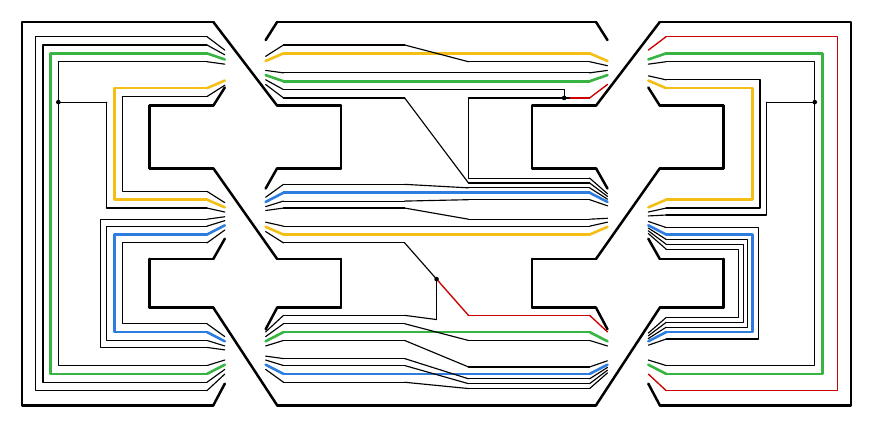}
    &
    \includegraphics[width=0.45\textwidth]{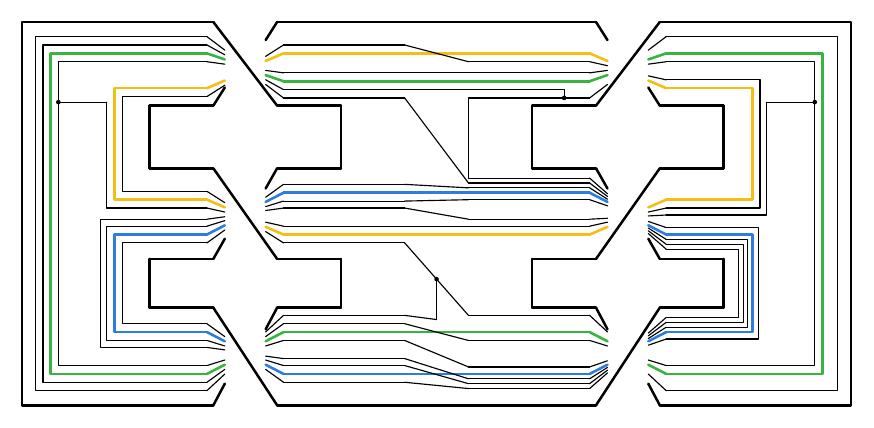}
    \\

    \includegraphics[width=0.45\textwidth]{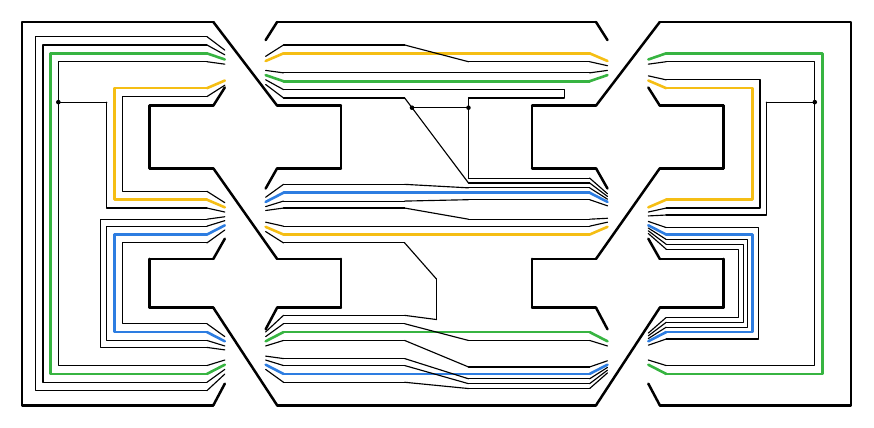}
    &
     \includegraphics[width=0.45\textwidth]{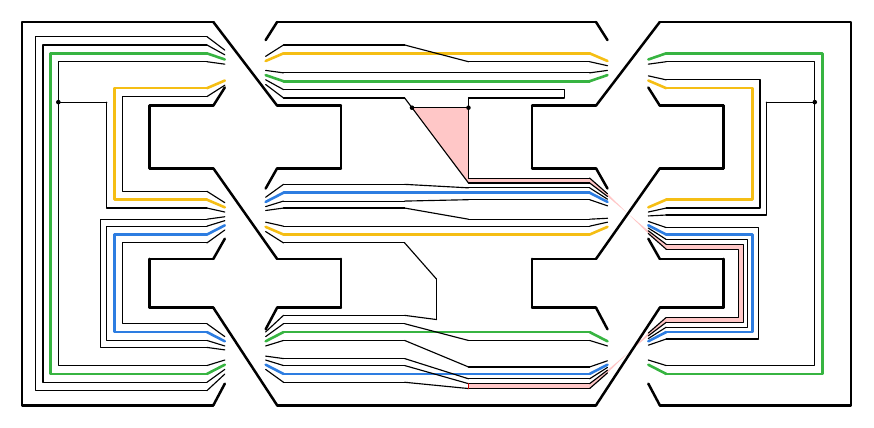}

    \\
       \includegraphics[width=0.45\textwidth]{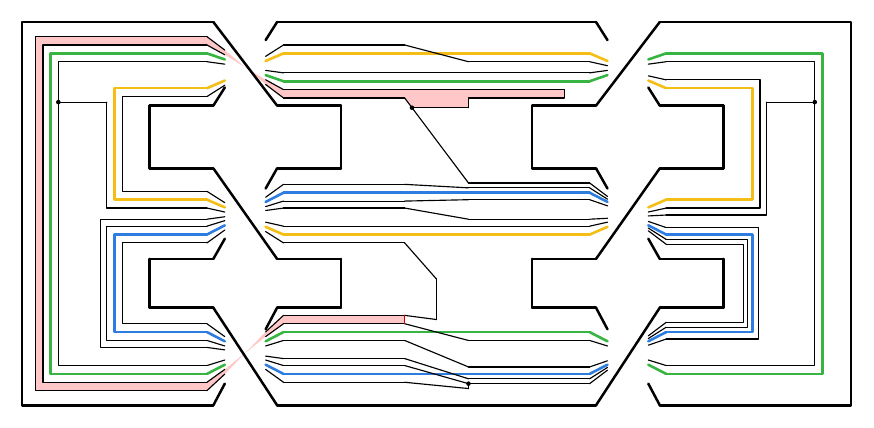}
    &
     \includegraphics[width=0.45\textwidth]{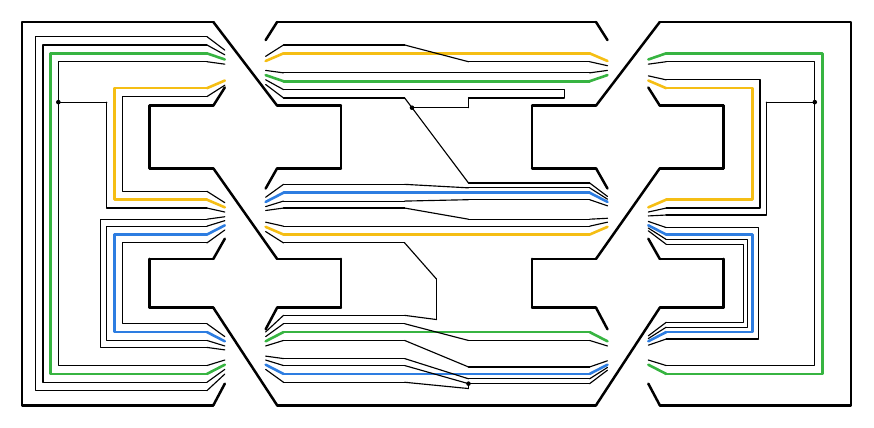}

    \\
        \includegraphics[width=0.45\textwidth]{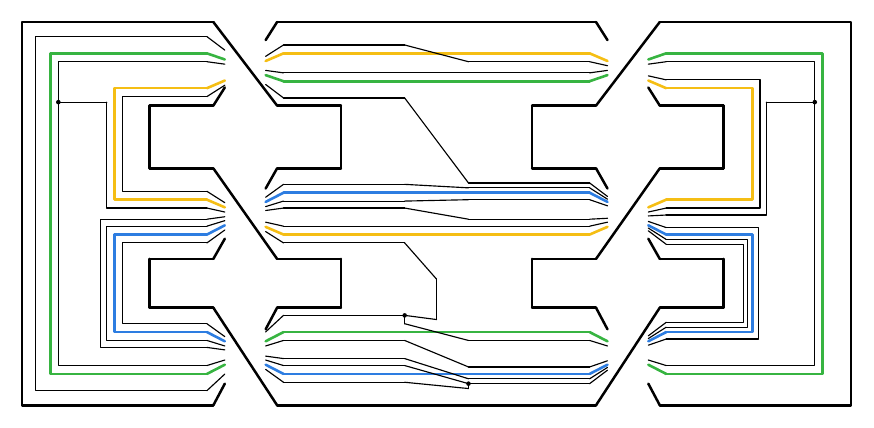}
    &
            \includegraphics[width=0.45\textwidth]{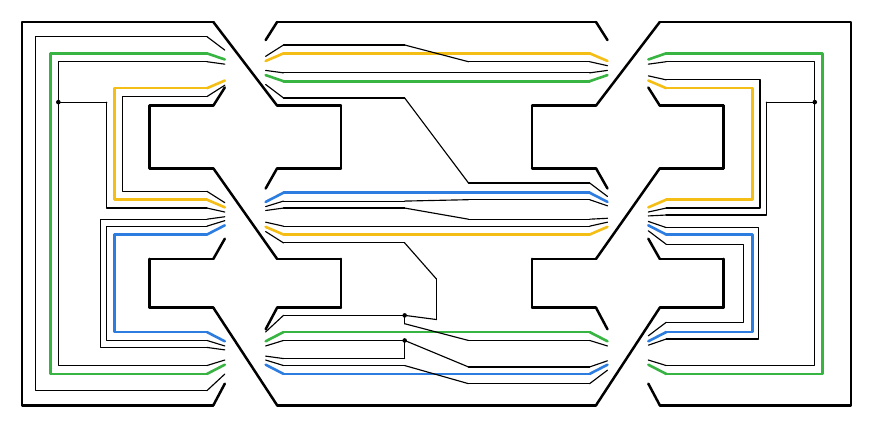}

    \\
};

\draw[operation arrow]
    (sequence-1-1.east) -- (sequence-1-2.west);

\draw[operation arrow]
    (sequence-1-2.south) -- (sequence-2-2.north);

\draw[operation arrow]
    (sequence-2-2.west) -- (sequence-2-1.east);

\draw[operation arrow]
    (sequence-2-1.south) -- (sequence-3-1.north);

\draw[operation arrow]
    (sequence-3-1.east) -- (sequence-3-2.west);
    \draw[operation arrow]
    (sequence-4-2.west) -- (sequence-4-1.east);

\draw[operation arrow]
    (sequence-3-2.south) -- (sequence-4-2.north);
    \draw[operation arrow]
    (sequence-4-1.south) -- (sequence-5-1.north);
    
    \draw[operation arrow]
    (sequence-5-1.east) --node[midway,above] {pass blue} (sequence-5-2.west);

\end{tikzpicture}
\end{fitdiagram}

\caption{A sequence of edge slides and isotopies for 4/7.}
\label{fig:4-7 operation-sequence2}

\end{figure}

\begin{figure}[p]
\centering

\begin{fitdiagram}
\begin{tikzpicture}[
  operation/.style={
    inner sep=0pt,
    outer sep=0pt,
    minimum width=0.40\textwidth,
    minimum height=0.23\textheight,
    anchor=center
  },
  operation arrow/.style={
    -{Stealth[length=3mm,width=2mm]},
    line width=0.8pt,
    shorten <=1mm,
    shorten >=1mm
  }
]

\matrix (sequence) [
  matrix of nodes,
  nodes={operation},
  column sep=16mm,
  row sep=10mm
] {
  \includegraphics[
    width=0.47\textwidth,
    height=0.23\textheight,
    keepaspectratio
  ]{4-7-step10.pdf}
  &
  \includegraphics[
    width=0.47\textwidth,
    height=0.23\textheight,
    keepaspectratio
  ]{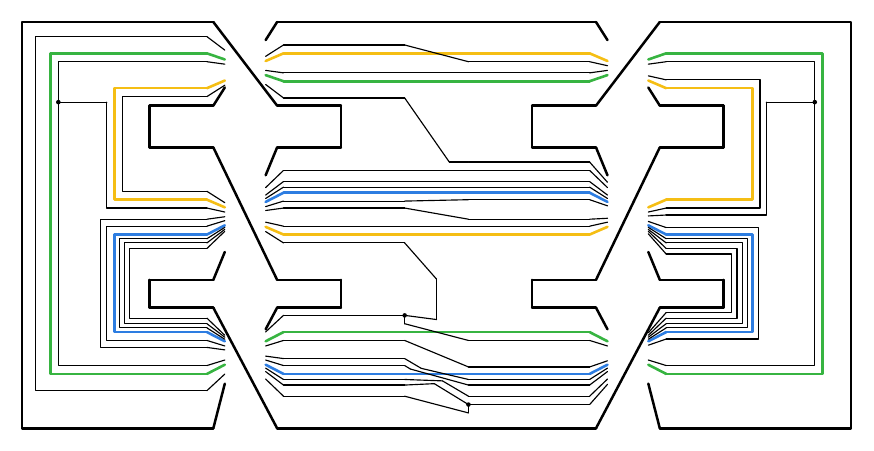}
  \\

  \includegraphics[
    width=0.47\textwidth,
    height=0.23\textheight,
    keepaspectratio
  ]{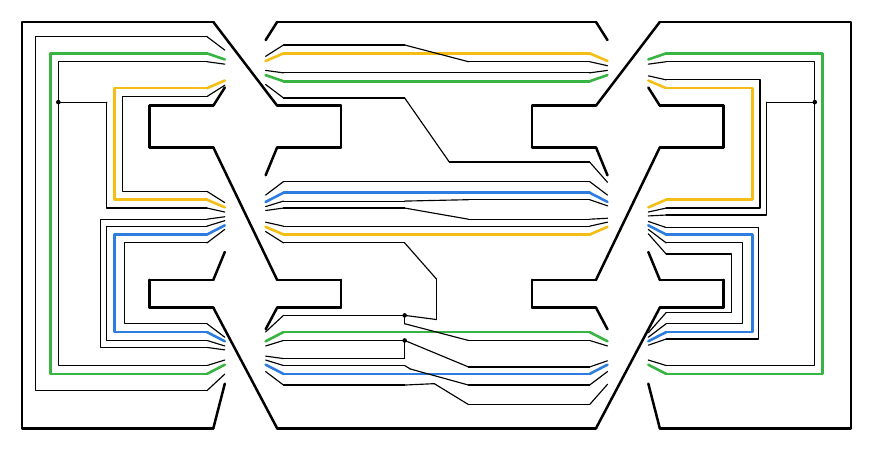}
  &
  \includegraphics[
    width=0.47\textwidth,
    height=0.23\textheight,
    keepaspectratio
  ]{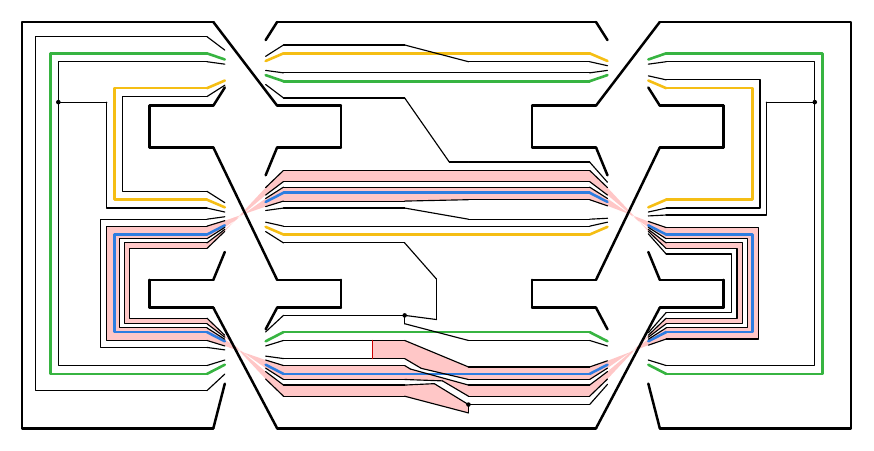}
  \\
};

\draw[operation arrow]
  (sequence-1-1.east) --node[midway, above]{twist} (sequence-1-2.west);

\draw[operation arrow]
  (sequence-1-2.south) --node[midway, right]{slide} (sequence-2-2.north);

\draw[operation arrow]
  (sequence-2-2.west) -- (sequence-2-1.east);


\end{tikzpicture}
\end{fitdiagram}

\caption{Passing the blue curve.}
\label{fig:4-7 first twist}

\end{figure}

\clearpage
\begin{figure}[p]
\centering

\begin{fitdiagram}
\begin{tikzpicture}[
  operation/.style={
    inner sep=0pt,
    outer sep=0pt,
    minimum width=0.40\textwidth,
    minimum height=0.23\textheight,
    anchor=center
  },
  operation arrow/.style={
    -{Stealth[length=3mm,width=2mm]},
    line width=0.8pt,
    shorten <=1mm,
    shorten >=1mm
  }
]

\matrix (sequence) [
  matrix of nodes,
  nodes={operation},
  column sep=16mm,
  row sep=10mm
] {
  \includegraphics[
    width=0.47\textwidth,
    height=0.23\textheight,
    keepaspectratio
  ]{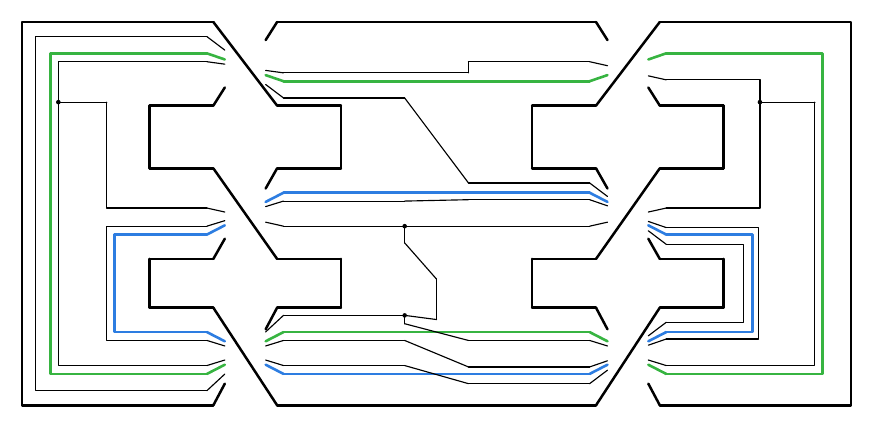}
  &
  \includegraphics[
    width=0.47\textwidth,
    height=0.23\textheight,
    keepaspectratio
  ]{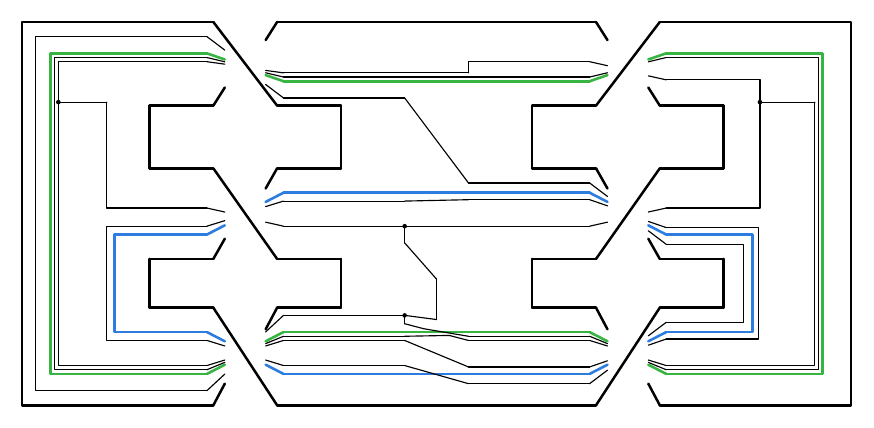}
  \\

  \includegraphics[
    width=0.47\textwidth,
    height=0.23\textheight,
    keepaspectratio
  ]{4-7-step16.pdf}
  &
  \includegraphics[
    width=0.47\textwidth,
    height=0.23\textheight,
    keepaspectratio
  ]{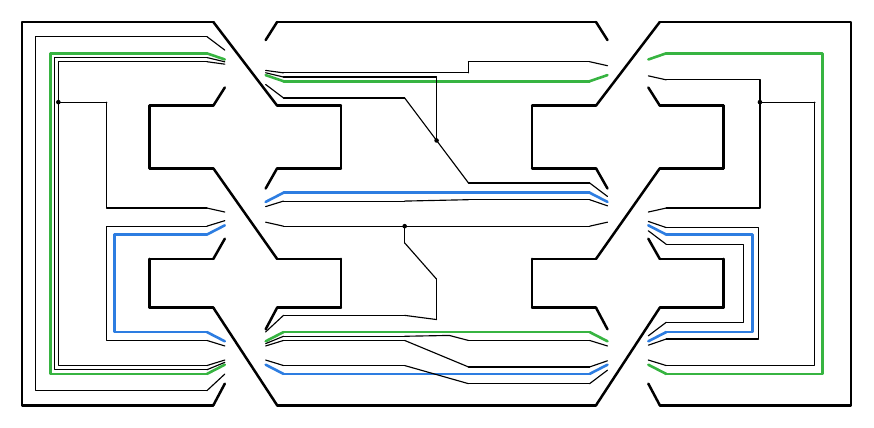}
  \\
};

\draw[operation arrow]
  (sequence-1-1.east) --node[midway, above]{twist} (sequence-1-2.west);

\draw[operation arrow]
  (sequence-1-2.south) --node[midway, right]{slide} (sequence-2-2.north);

\draw[operation arrow]
  (sequence-2-2.west) --node[midway,above]{slide} (sequence-2-1.east);


\end{tikzpicture}
\end{fitdiagram}

\caption{Passing the green curve.}
\label{fig:4-7 third twist}

\end{figure}

\clearpage

\begin{figure}[p]
\centering

\begin{fitdiagram}
\begin{tikzpicture}[
    operation/.style={
        inner sep=0pt,
        outer sep=0pt
    },
    operation arrow/.style={
        -{Stealth[length=3mm,width=2mm]},
        line width=0.8pt,
        shorten <=3mm,
        shorten >=3mm
    }
]

\matrix (sequence) [
    matrix of nodes,
    nodes={operation},
    column sep=16mm,
    row sep=10mm
] {
    \includegraphics[width=0.45\textwidth]{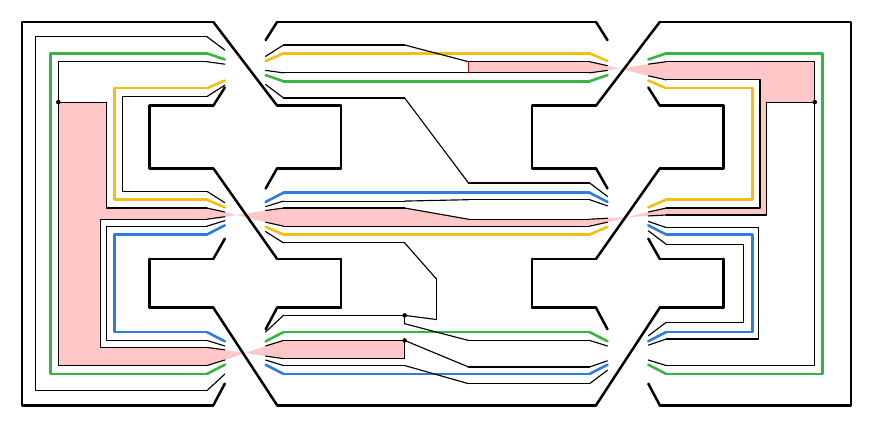}
    &
    \includegraphics[width=0.45\textwidth]{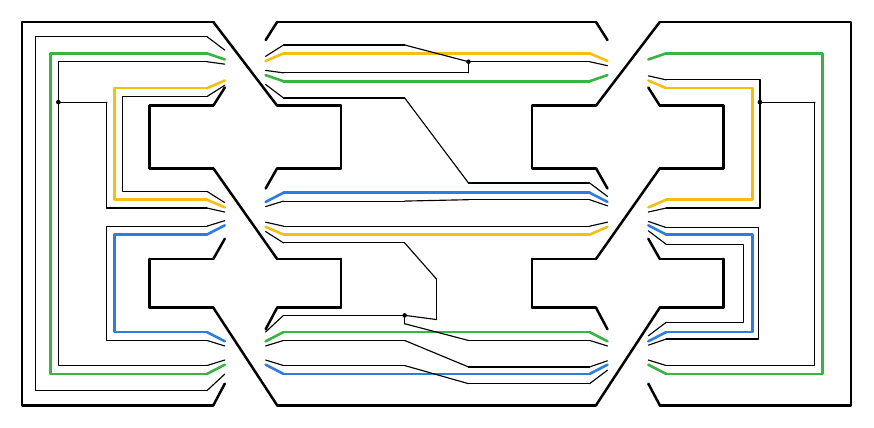}
    \\

    \includegraphics[width=0.45\textwidth]{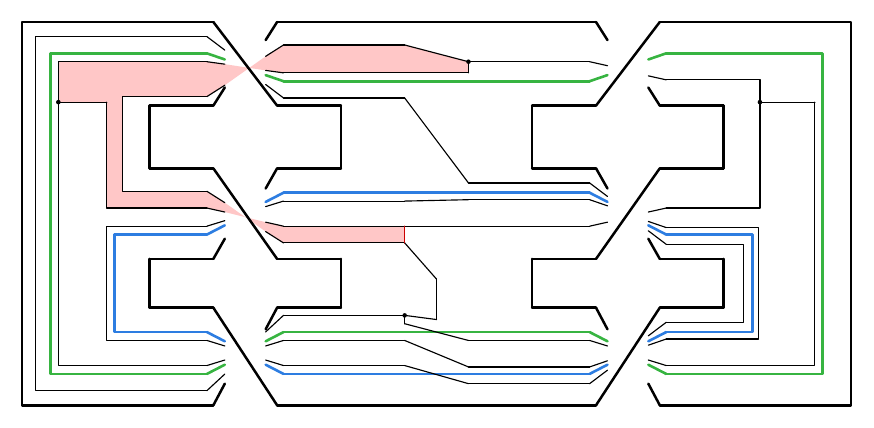}
    &
    \includegraphics[width=0.45\textwidth]{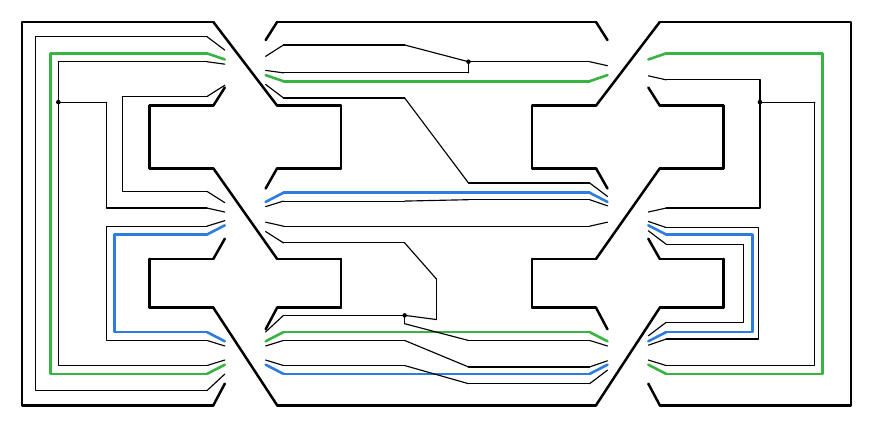}
    \\

    \includegraphics[width=0.45\textwidth]{4-7-step16.pdf}
    &
     \includegraphics[width=0.45\textwidth]{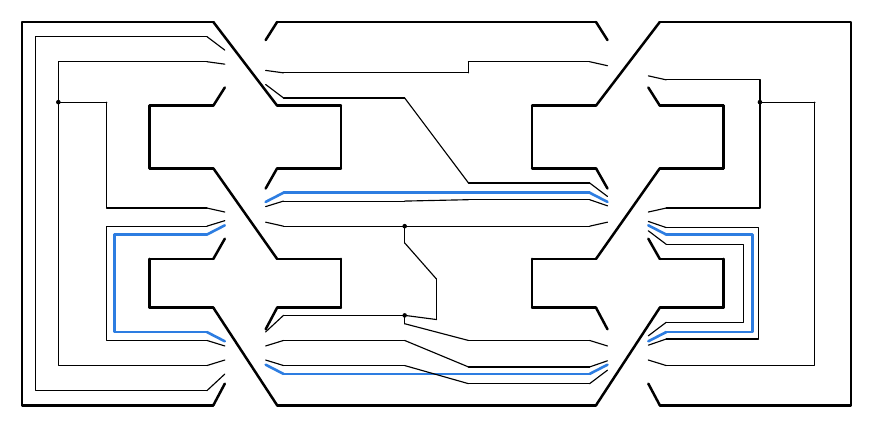}

    \\
       \includegraphics[width=0.45\textwidth]{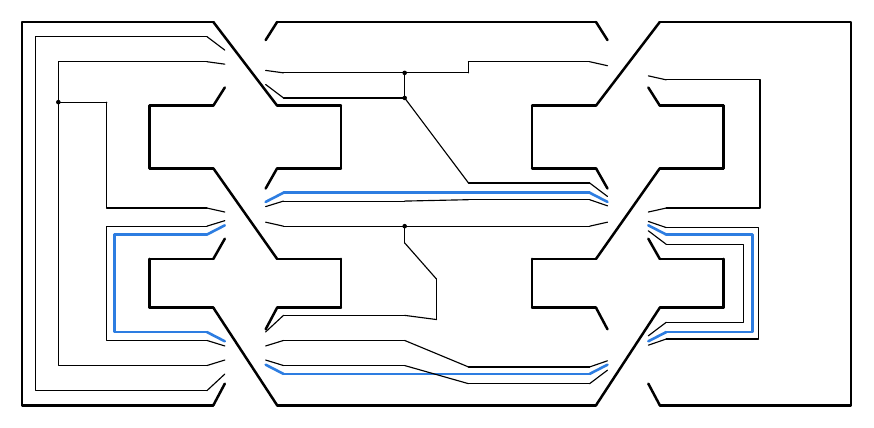}
    &
     \includegraphics[width=0.45\textwidth]{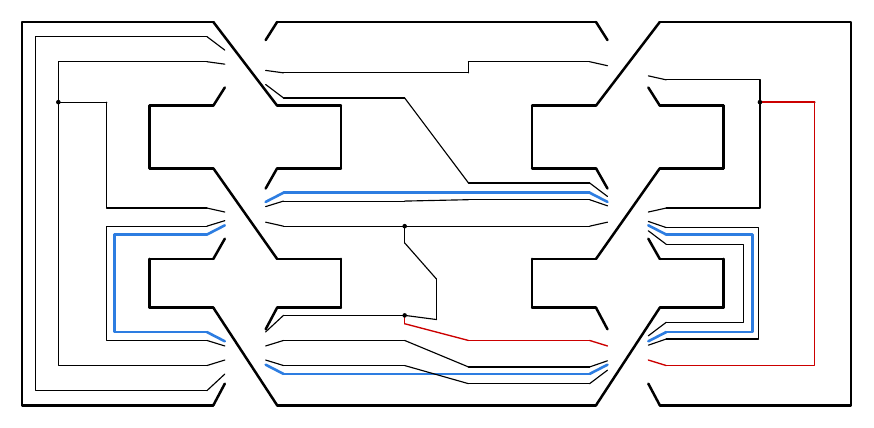}

    \\
        \includegraphics[width=0.45\textwidth]{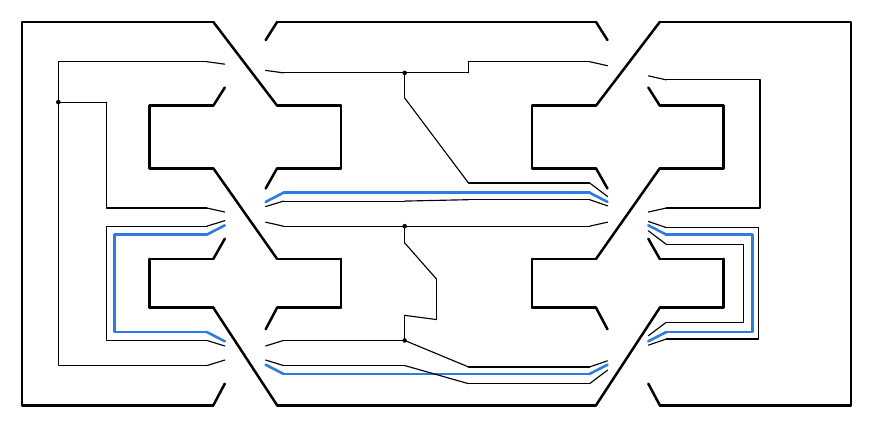}
    &
            \includegraphics[width=0.45\textwidth]{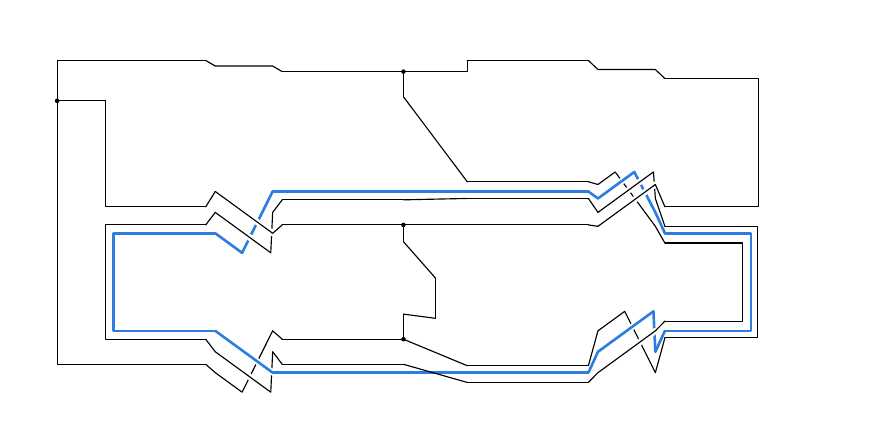}

    \\
};

\draw[operation arrow]
    (sequence-1-1.east) --(sequence-1-2.west);

\draw[operation arrow]
    (sequence-1-2.south) --node[midway, right]{pass yellow}  (sequence-2-2.north);

\draw[operation arrow]
    (sequence-2-2.west) -- (sequence-2-1.east);

\draw[operation arrow]
    (sequence-2-1.south) -- (sequence-3-1.north);

\draw[operation arrow]
    (sequence-3-1.east) --node[midway, above]{pass green}  (sequence-3-2.west);
    \draw[operation arrow]
    (sequence-4-2.west) -- (sequence-4-1.east);

\draw[operation arrow]
    (sequence-3-2.south) -- (sequence-4-2.north);
    \draw[operation arrow]
    (sequence-4-1.south) -- (sequence-5-1.north);
    
    \draw[operation arrow]
    (sequence-5-1.east) -- (sequence-5-2.west);

\end{tikzpicture}
\end{fitdiagram}

\caption{A sequence of edge slides and isotopies for 4/7.}
\label{fig:4-7operation-sequence3}

\end{figure}

\begin{figure}[htbp]
\centering

\begin{fitdiagram}
\begin{tikzpicture}[
  operation/.style={
    inner sep=0pt,
    outer sep=0pt,
    minimum width=0.40\textwidth,
    minimum height=0.23\textheight,
    anchor=center
  },
  operation arrow/.style={
    -{Stealth[length=3mm,width=2mm]},
    line width=0.8pt,
    shorten <=1mm,
    shorten >=1mm
  }
]

\matrix (sequence) [
  matrix of nodes,
  nodes={operation},
  column sep=16mm,
  row sep=10mm
] {
  \includegraphics[
    width=0.40\textwidth,
    height=0.23\textheight,
    keepaspectratio
  ]{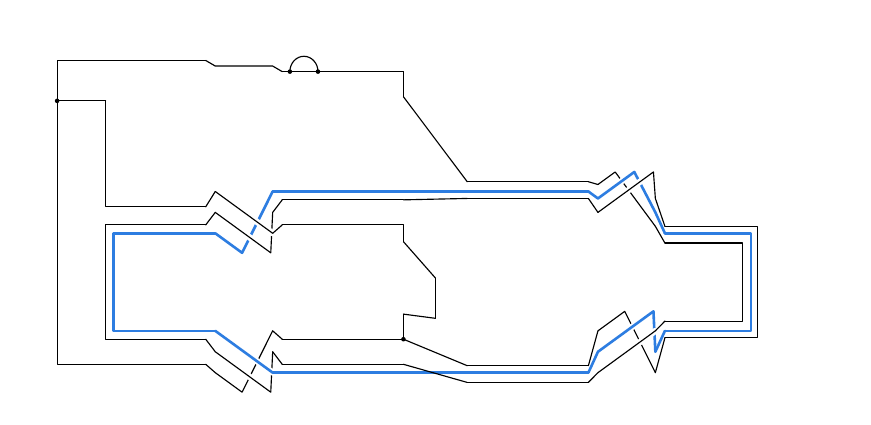}
  &
  \includegraphics[
    width=0.40\textwidth,
    height=0.23\textheight,
    keepaspectratio
  ]{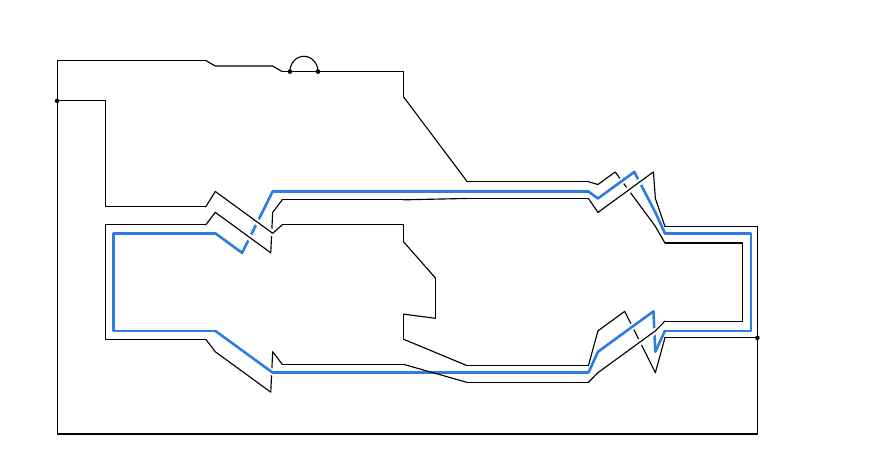}
  \\

  \includegraphics[
    width=0.40\textwidth,
    height=0.23\textheight,
    keepaspectratio
  ]{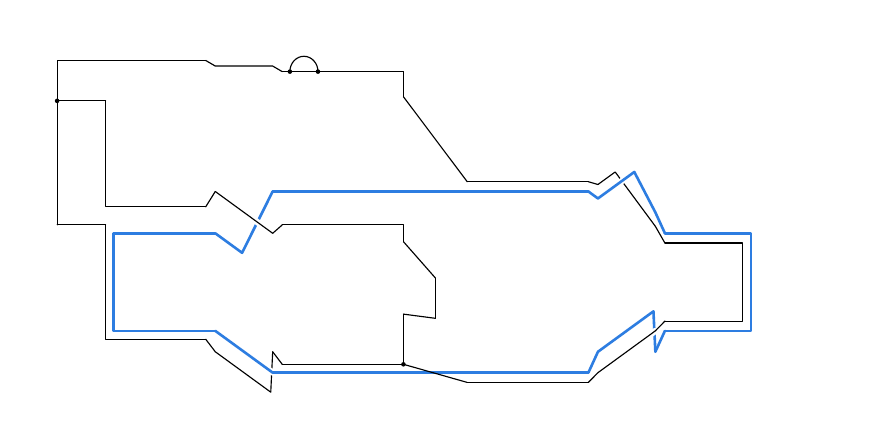}
  &
  \includegraphics[
    width=0.40\textwidth,
    height=0.23\textheight,
    keepaspectratio
  ]{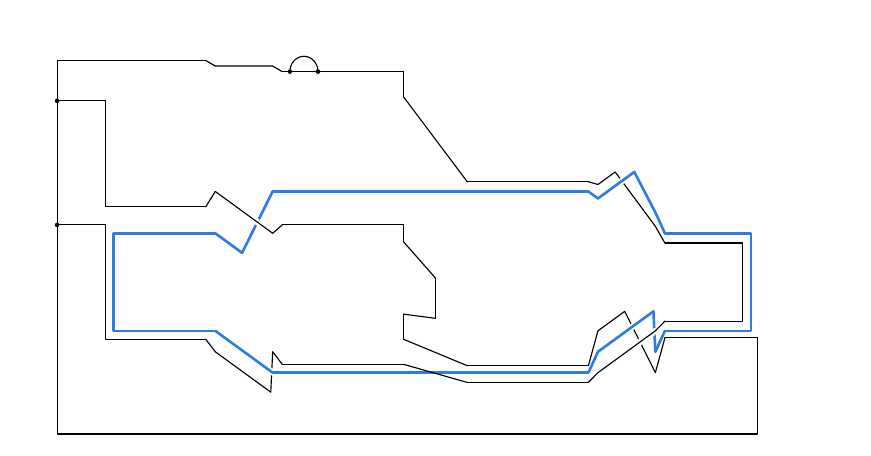}
  \\

  \includegraphics[
    width=0.40\textwidth,
    height=0.23\textheight,
    keepaspectratio
  ]{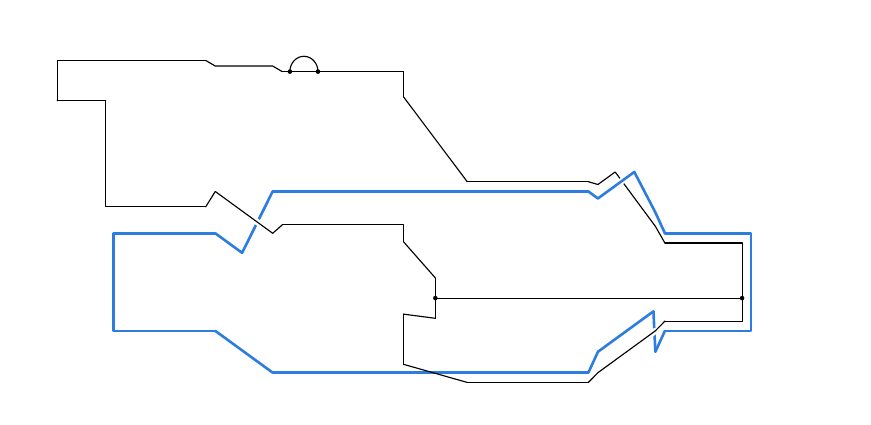}
  &
  \includegraphics[
    width=0.30\textwidth,
    height=0.23\textheight,
    keepaspectratio
  ]{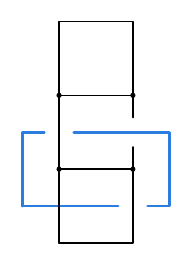}
  \\
};

\draw[operation arrow]
  (sequence-1-1.east) -- (sequence-1-2.west);

\draw[operation arrow]
  (sequence-1-2.south) -- (sequence-2-2.north);

\draw[operation arrow]
  (sequence-2-2.west) -- (sequence-2-1.east);

\draw[operation arrow]
  (sequence-2-1.south) -- (sequence-3-1.north);

\draw[operation arrow]
  (sequence-3-1.east) -- (sequence-3-2.west);

\end{tikzpicture}
\end{fitdiagram}

\caption{A sequence of edge slides and isotopies for $4/7$.}
\label{fig:4-7operation-sequence4}

\end{figure}

\clearpage
\nocite{*}
\bibliographystyle{amsplain}
\bibliography{references}

\end{document}